\documentclass[letterpaper,12pt]{article}
\usepackage{amsmath,amsthm,amssymb,bm}
\usepackage{graphicx,subfigure}
\usepackage{enumerate,arydshln,multirow}
\usepackage{hyperref}
\usepackage[authoryear,longnamesfirst,sectionbib]{natbib}
\usepackage[left=3.5cm,right=3.5cm,top=3.5cm,bottom=3.5cm]{geometry}
\newtheorem{theorem}{Theorem}

\newtheorem{lemma}[theorem]{Lemma}
\newtheorem{corollary}[theorem]{Corollary}
\newtheorem{hyp}[theorem]{Assumption}
\newtheorem{remark}[theorem]{Remark}
\newcommand{\1}[1]{\bm{1}_{\{#1\}}}
\def\esp{\mathbb E}
\def\pr{\mathbb P}

\def\var{\mathrm{var}}
\def\Nset{\mathbb N} 
\def\Rset{\mathbb R}
\def\mbw{\mathbb W}

\def\mcf{\mathcal F}
\def\mcj{\mathcal J}
\def\Vset{\mathcal X}

\def\NOEset{\mathcal E}

\def\omle{\hat \theta_N}
\def\wlsemodtheta{\hat\theta_N^W}
\def\mcemodtheta{\hat\theta_N^\mbw}
\def\fisher{\mathcal{I}_n}

\def\suite{\mathcal{S}}
\def\total{Q}
\def\dernier{T}
\newcommand{\citedef}{\citep[see~][]}
\title{Parameter estimation of a two-colored urn model class}
\author{Line Le Goff \and Philippe Soulier}

\begin{document}
\maketitle
\begin{abstract}
  Though widely used in applications, reinforced random walk on graphs have never been the subject
  of a valid statistical inference. We develop in this paper a statistical framework for a general
  two-colored urn model. The probability to draw a ball at each step depends on the number of balls
  of each color and on a multidimensional parameter~$\theta$ through a function~$f$, called a choice
  function. We introduce two estimators of~$\theta$: the maximum likelihood estimator and a weighted
  least squares estimator which is less efficient, but is closer to the calibration techniques used
  in the applied literature. In general, the model is an inhomogeneous Markov chain and this
  property makes the estimation of the parameter impossible on a single path, even if it were
  infinite. Therefore we assume that we observe i.i.d.~experiments, each of a predetermined finite
  length. This is coherent with the usual experimental set-ups. We apply the statistical framework
  to a real life experiment: the selection of a path among pre-existing channels by an ant
  colony. We performed experiments, which consisted of letting ants pass through the branches of a
  fork. We consider the particular urn model proposed by J.-L.~Deneubourg in 1990 to describe this
  phenomenon. We simulate this model for several parameter values in order to assess the accuracy of
  the MLE and the WLSE. Then we estimate the parameter from the experimental data and evaluate
  confident regions with Bootstrap algorithms. The findings of this paper do not contradict the
  biological literature, but give statistical significance to the values of the parameter found
  therein.
\end{abstract}
\clearpage

\section{Introduction}

Urn models have been studied for nearly one century. In~1931, G.~P\'olya provided the first probabilistic result on the game consisting in drawing a ball from an urn initially containing one red ball and one black ball~\citedef{Polya1931}. At each time step, a ball is drawn and put back in the urn with an additional ball of the same color. The probability to draw a red ball is the proportion of red balls in the urn. G.~P\'olya proved that, as the number of draws tends to infinity, the proportion of red balls tends to a random variable following the uniform distribution on $[0,1]$.

The P\'olya urn is easily generalizable to a large class of two-colored urn models characterized by a choice function $f:\Theta\times\Nset\times\Nset\to[0,1]$ which itself depends on a parameter~${\theta\in\Theta\subset\Rset^d}$, ${d\le1}$. Let~$R_n$ and~$B_n$ be the numbers of red and black balls in the urn after~$n$~draws. Note that, by construction,~${R_n+B_n = n}$.  The
probability that the ${(n+1)}$-th ball is red is given by
\begin{align}
  \label{eq:model-general}
  prob_{n+1}^R = f(\theta, R_n,B_n) = f(\theta, R_n,n-R_n) \; .
\end{align}
Consequently, the probability to draw a black ball
at time~$n+1$ is $1-prob_{n+1}^R$.

The goal of this paper is to propose a valid statistical methodology to estimate the
parameter~$\theta$. We suppose that we observe~$N$ independent paths, each consisting of a sequence
of $n$~colors drawn by the model defined in~\eqref{eq:model-general}. The statistical theory is
developed as~$n$~is fixed and~$N$ tends to infinity. We choose this framework since in some models,
it is not possible to obtain consistent estimators of the parameter with only one path, even if its
length $n$ increases to infinity. 
Moreover data from real experiments comprises a set of finite paths. We define estimators
for~$\theta$ that we prove to be consistent and asymptotically  normal under some usual regularity
assumptions on the model~\eqref{eq:model-general}. We study more precisely two particular cases: the
maximum likelihood estimator (MLE) and the weighted least squares estimators (WLSE).

We have applied these statistical tools to the problem of path formation by an ant colony. One of the
fundamental factors affecting an organism's survival is its ability to optimally and dynamically
exploit its environment. For example, in order to take advantage of the best sites of resources,
housing or reproduction, these areas must be discovered and exploited at the earliest opportunity.
Many species of ants rise to this challenge by developing a network of paths, which connects
different strategic sites such as nests and food sources. These paths consist of pheromones,
attractive chemical substances. We focus on a specific aspect of this phenomenon: the selection of a
path among pre-existing channels. When exploring their environment, ants often face bifurcations and
must bypass obstacles. As shown by experimental studies, the laying of pheromones by ants passing
successively through a bifurcation results in two possible outcomes: either one branch is eventually
selected and the other abandoned, or both branches end up being uniformly chosen \citedef{Deneubourg1990}.
 The analysis of the spontaneous path formation by a colony of ants is made
difficult by the absence of any means to measure precisely the quantity of, or even detect, the
pheromones laid by the ants.

It is commonly assumed that when approaching a bifurcation, ants choose a branch and lay a certain
constant amount of pheromone without ever turning back. Consequently, the quantity of pheromone laid
on a branch is proportional to the number of ants which passed through it. Thus this phenomenon can
be described by a urn model as proposed by J.-L.~Deneubourg et
al. in~1990 \citedef{Deneubourg1990}. More precisely they define a choice function with a two-dimension
parameter $(\alpha,c)\in(0,1)^2$ such that the probability for an ant to choose the right branch
after~$R_n$~passages through the right branch and~$n$~passages in total is given by
\begin{equation}
\label{eq:model-bio}
Prob_{n+1}^R = \frac{(c+R_n)^\alpha}{(c+R_n)^\alpha+(c+n-R_n)^\alpha}\;.
\end{equation}
The probability to choose the left branch is consequently $1-Prob_{n+1}^R$. The
parameter~$\alpha$~makes this model non linear with respect to the proportion of passages through
one branch. It models the sensitivity of the ant to the concentration of pheromone. The parameter
$c$ is the intrinsic attractiveness of each branch and can also be interpreted as the inverse of the
attractiveness (or the strength) of the pheromone deposit laid by each
ant.

Several probabilistic studies provide the asymptotic behavior of~$R_n/n$ in terms of~$\alpha$
and~$c$~\citedef{Polya1931,Tarres2011,Davis1990}. The influence of the two parameters is on different
time scales, but they can contribute to the same effect (selection of one branch or unifomization of
the traffic on the two branches) or have antagonistic effects. The model is thus characterized by
four phases, according to the values of $\alpha$ and $c$: slow or fast uniformization; slow or fast
selection. The phase most commonly considered in the literature is slow selection. This corresponds
in our model to~$\alpha$ and $c$ larger than~$1$: selection of one branch will eventually happen,
though slowly because of weak pheromone deposits
\citedef{Deneubourg1990,Vittori2006,Garnier2009}. However, the model may account for other
possibilities, such as fast uniformization, which occurs when $\alpha<1$ and~$c>1$. One purpose of
this paper is to investigate more thoroughly these possibilities which have been more or less
overlooked in the previous literature.

Ethological studies have already provided values for the parameter~$(\alpha,c)$, but the methods
used mainly consisted of calibration without control of the statistical validity of these methods
and results \citedef{Deneubourg1990, Vittori2006, Garnier2009, Thienen2014} and in particular these
methods do not produce confidence regions. However this type of information supplies interesting and
important elements to the behavioral discussion. However this type of information supply interesting
and important elements to the behavioral discussion. In this paper, we define the MLE and the WLSE
for the ant behavior model~\eqref{eq:model-bio}. We assess the quality of these estimators in a
simulation experiment. We then use them on experimental data provided by a real life experiment
performed to this purpose with ants. We also compute confidence region by a Bootstrap algorithm.

The model~\eqref{eq:model-bio} can be applied, {\it mutatis mutandis}, to many other fields. For
instance we can consider a fork with two branches as a neuron having two axons. During each period
of time, the length of one of the two axons increases. The longer an axon is, the higher the
probability that it will further grow. Thus the dynamic of this biological system may also be
modeled by~\eqref{eq:model-bio} \citedef{Khanin2001}. Furthermore, using the notion of a choice
function (see Section~\ref{sec:generalization}), the model~\eqref{eq:model-general} can be adapted
and applied to many situations where a binary choice occurs, or to even more complex situations such
as networks with several nodes \citedef{Jeanson2003,Pemantle2007}, for instance network exploration by
an ant colony \citedef{Aron1990, Beckers1993, Nicolis1999, Dussutour2005, Nicolis2008, Thienen2014,
  Arganda2014}.

The paper is organized as follows. The model~\eqref{eq:model-general} and the statistical framework
is rigorously defined in Section~\ref{sec:estimation}. Section~\ref{sec:MLE} is focused on the MLE
and Section~\ref{sec:wlse} on the WLSE. Under some usual regularity conditions, we prove that both
estimators are consistent and asymptotically normal. Moreover the MLE is asymptotically
efficient. The numerical implementation of the MLE may be difficult and unstable (and lengthy),
therefore, a WLSE is considered and theoretically studied. This estimator does not match the
theoretical performances of the MLE, but is easier to compute and is popular among
practitioners. Section~\ref{sec:generalization} proposes an extension of the general urn model to a
class of vertex reinforced random walk on graphs. Section~\ref{sec:bio-theory} is an adaptation of
the statistical framework to the ethological problem. In Section~\ref{sec:theorie-interpretation},
we first introduce the assumptions on the ant behavior and then we define the model proposed by
J.-L.~Deneubourg. The four phases of the model are described more precisely and are interpreted from
the point of view of ethology. Section~\ref{sec:parametric} rewrites the MLE and the WLSE for this
particular case. In Section~\ref{sec:fisher-single-path}, we show that it is not possible to
consistently estimate the parameter of the model on a single experiment of length $n$, even if $n$
tends to infinity. In order to assess the performance of the estimators, a short simulation
experiment is reported in Section~\ref{sec:simulation}. Section~\ref{sec:ants} reports the study on
the experimental data. The experimental protocol and the data produced are described in
Sections~\ref{sec:exp_desc} and~\ref{sec:data_repres}. Section~\ref{sec:para_est} supplies the
estimation results (computation of the estimators and their confidence regions). We provide some
concluding remarks in Section~\ref{sec:concluding} and prove the theoretical statistical results of
this paper in Section~\ref{sec:proofs}.

\section{Parameter estimation}
\label{sec:estimation}

Let us first write precisely the model studied and statistical framework used. We assume that $\{X_k,k\geq1\}$ is a sequence of Bernoulli random variables (representing the colors of the balls drawn:~1~for red and~0~for black) and that there exists a function $f_0:\Nset^2\to[0,1]$ such that, for all integers $k$,
\begin{align*}
\pr(X_{k+1} = 1 \mid \mcf_k) = f_0(Z_k,k-Z_k) \; .
\end{align*}
where $Z_0=0$ and for $k\geq1$, $Z_k=X_1+\cdots+X_k$ and $\mcf_k$ is the sigma-field generated by $Z_0,X_1,\dots,X_k$. The random walk  $\{Z_k,k\geq0\}$ is an inhomogeneous Markov chain. For~${n\geq1}$ and a sequence $(e_1,\dots,e_n)\in\{0,1\}^n$, applying the Markov property,  we obtain
\begin{multline}
\label{eq:vrais}
\pr(X_1=e_1,\dots, X_n=e_n) = \prod_{k=0}^{n-1} f_0(z_k,k-z_k)^{e_{k+1}} \{1-f_0(z_k,k-z_k)\}^{1-e_{k+1}} \; ,
\end{multline}
where $z_k = e_1 + \cdots + e_k$ with $z_0=0$, by convention.

We assume that we observe $N$ experiments, each consisting in a path of length~$n$ of~the  model~(\ref{eq:vrais}). In other words, we have a set of $N$ sequences of~$n$~consecutive draws. For~${j=1,\dots,N}$ and~${k=1,\dots,n}$, let~${X_k^j \in\{0,1\}}$ denote the color of the $k$-th ball drawn in the $j$-th experiment. Let $Z_0^j=0$ and $Z_k^j=\sum_{i=1}^k X_i^j$,~${k\geq1}$ be the total number of red balls drawn at time $k$ during the $j$-th experiment, so that $X_k^{j}=Z_k^j-Z_{k-1}^j$.  In all the paper, $n$ will be fixed and our asymptotic results will be obtained with $N$ (the number of experiments) tending to $\infty$.

Let $\Theta\subset\Rset^d$ and $f:\Theta\times \Nset^2\to (0,1)$ be a function, called the {\it choice function} of the urn. We assume that there exists $\theta_0\in\Theta$ such that ${f_0(\cdot,\cdot)=f(\theta_0,\cdot,\cdot)}$, i.e.~for~${k=0,\dots,n-1}$ and~${i=0,\dots,k}$,
\begin{align*}
\pr(X_{k+1}^j = 1 \mid Z_{k}^j=i) = f(\theta_0,i,k-i) \; .
\end{align*}

To proof the consistency and the asymptotic normality of the estimators introduced below, we need the following assumptions on the choice function~$f$. For any function~$g$~defined on $\Theta$, we denote $\dot{g}$ and $\ddot{g}$ the gradient and Hessian matrix with respect to $\theta$, $\partial_sg$~the partial derivative with respect to the~\mbox{$s$-th}~component~$\theta_s$ of $\theta$, $1\leq s \leq d$ and $A'$ the transpose of the vector or matrix~$A$.

\begin{hyp}
\label{hypo:modele-regulier}
\ \begin{enumerate}[(i)]
	\item \label{item:regularite} (Regularity) The set $\Theta$ is a compact with non empty interior. For $0 \leq i \leq k \leq n-1$, $f_0(i,k-i)>0$ and the function~${\theta \rightarrow f(\theta,i,k-i)}$ is twice continuously differentiable on $\Theta$.
	\item \label{item:identifiability} (Identifiability) If $f(\theta_1,i,k-i)=f(\theta_2,i,k-i)$ for all $0 \leq i \leq k\leq n-1$, then $\theta_1=\theta_2$,
	\item \label{item:invertibilite} The $n(n-1)$-dimensional vectors $\{\partial_s f(\theta_0,i,k-i),0\leq i \leq k \leq n-1\}$, $1 \leq s \leq d$, are linearly independent in $\Rset^{n(n-1)}$.
	\end{enumerate}
\end{hyp}
Assumption~\ref{hypo:modele-regulier} ensures that $\theta_0$ is the unique maximizer of $L$ and that the Fisher information matrix
\begin{align*}
\fisher(\theta_0) = - \ddot{L}(\theta_0) = \sum_{k=0}^{n-1} \sum_{i=0}^k \frac{\pr(Z_k=i)}{f_0(i,k-i)\bar f_0(i,k-i)} \dot f_0 (i,k-i) \dot f_0(i,k-i)'
\end{align*}
is invertible, where we denote $\dot{f}_0(i,k-i) = \dot{f}(\theta_0,i,k-i)$ and $\bar f_0(i,k-i)=1-f_0(i,k-i)$. The explicit expression of the probabilities $\pr(Z_k=i)$, $i,k\in\Nset$, is given in Section~\ref{sec:P(Zk=i)}, Equation~\eqref{eq:proba-Zk}.

\subsection{Maximum likelihood estimation (MLE)}
\label{sec:MLE}

The structure of the model~\eqref{eq:vrais} allows to have an explicit expression of likelihood~${V_N(\theta)}$. The independence of the $N$ experiments yields the following multiplicative form 
\begin{eqnarray*}
V_N(\theta) = \prod_{j=1}^N \prod_{k=0}^{n-1} f(\theta,Z_k^j,k-Z_k^j)^{X^j_{k+1}}\{1-f(\theta,Z_k^j,k-Z_k^j)\}^{1-X_{k+1}^j} \; .
\end{eqnarray*}
The log-likelihood function $L_N$ based on $N$ paths, is thus given by
\begin{multline}
\label{eq:def-LN}
L_N(\theta) = \sum_{j=1}^N \sum_{k=0}^{n-1} \big\{ X^j_{k+1} \log f(\theta,Z^j_{k},k-Z_k^j)\\
 + (1-X^j_{k+1}) \log \{1-f(\theta,Z^j_{k},k-Z_k^j)\} \big\} \; .
\end{multline}
Let $\omle$ be the maximum likelihood estimator of~$\theta_0$, that is
\begin{equation}
\label{eq:def-mle}
\omle = \arg\max_{\theta\in\Theta} L_N(\theta) \; .
\end{equation}
Define $L(\theta)=N^{-1} \esp[L_N(\theta)]$ (where the dependence in $n$ is omitted). Then, 
\begin{multline}
\label{eq:def-L}
L(\theta) = \sum_{k=0}^{n-1} \esp\big[ f_0(Z_{k},k-Z_{k}) \log f(\theta,Z_k,k-Z_{k})\\
+ \{1-f_0(Z_k,k-Z_{k})\}\log\{1-  f(\theta,Z_{k},k-Z_{k})\} \big] \; .
\end{multline}

Let $\mathcal{N}(m,\Sigma)$ denote the Gaussian distribution with mean $m$ and covariance~$\Sigma$.
\begin{theorem}
\label{theo:consistence-clt-emv}
If Assumptions~\ref{hypo:modele-regulier}-\eqref{item:regularite} and~\ref{hypo:modele-regulier}-\eqref{item:identifiability} hold then the maximum likelihood estimator~$\omle$ is a strongly consistent estimator of $\theta_0$. If moreover Assumption~\mbox{\ref{hypo:modele-regulier}-\eqref{item:invertibilite}} holds and~$\theta_0$~is an interior point of~$\Theta$, then, as~$N$ tends to~$\infty$, $\sqrt{N} (\omle-\theta_0)$ converges weakly towards~${\mathcal{N}(0,\fisher^{-1}(\theta_0))}$.
\end{theorem}
The proof is in Section~\ref{sec:proof-theo-estim}. It is the consequence of a more general result stated and proved therein.

\subsection{Weighted least squares estimation (WLSE)} 
\label{sec:wlse}

Least squares estimators are very popular among practitioners. Moreover for some urn models, the MLE may be numerically unstable hence difficult (and lengthy) to compute so the WLSE constitutes a convenient alternative. In order to describe this estimator, we introduce some notation. For~${0 \leq i \leq k \leq n-1}$, define
\begin{eqnarray} {a}_N(i,k-i) = \frac1N \sum_{j=1}^N \1{Z_k^{j}=i} \;\; \mbox{ and } \;\; {p}_N(i,k-i) = \frac{\frac1N \sum_{j=1}^N \1{Z_k^{j}=i}X_{k+1}^j}{{a}_N(i,k-i)} \; , \label{eq:empiriques}
\end{eqnarray}
with the convention~$\frac00=0$. The quantity ${a}_N(i,k-i)$ is the empirical probability that~$i$ red balls have been drawn at time $k$ and ${p}_N(i,k-i)$ is the empirical conditional probability that a red ball is again chosen at time~$k+1$ given $i$ red balls were drawn at time
$k$.

We further define ${q}_N(i,k-i)=1-{p}_N(i,k-i)$ and $\bar{f}(\theta,i,k-i)=1-f(\theta,i,k-i)$. Let~$\{w_N(i,k-i), 0\leq i \leq k\}$ be a sequence of weights and define the contrast function
\begin{equation}
\label{eq:def-WN}
W_N(\theta) = \sum_{k=0}^{n-1} \sum_{i=0}^k w_N(i,k-i) \{p_N(i,k-i)-f(\theta,i,k-i)\}^2 \; .
\end{equation}
The weighted least squares estimator minimizes $W_N$, that is
\begin{equation}
\label{eq:def-wlse}
\wlsemodtheta = \arg \min_{\theta\in\Theta} W_N(\theta) \; .
\end{equation}

\begin{theorem}
\label{theo:wlse}
If Assumptions~\ref{hypo:modele-regulier}-\eqref{item:regularite} and~\ref{hypo:modele-regulier}-\eqref{item:identifiability} hold and if the weights $w_N$ converge almost surely to a sequence of positive weights $w_0$, then the weighted least squares estimator $\wlsemodtheta$ is a strongly consistent estimator of $\theta_0$.

If moreover Assumption~\ref{hypo:modele-regulier}-\eqref{item:invertibilite} holds and $\theta_0$ is an interior point of~$\Theta$, then as~$N$ tends to~$\infty$, $\sqrt{N} (\wlsemodtheta-\theta_0)$ converges weakly towards $\mathcal{N}(0,\Sigma_n(\theta_0))$, where~${\Sigma_n(\theta_0)}$ is a definite positive covariance matrix.

If moreover $w_N(i,k-i)={p_N}^{-1}{q_N}^{-1} a_N(i,k-i)$, for all ${0\le i \le k\le n-1}$, the estimator~$\wlsemodtheta$~is asymptotically efficient, i.e. $\Sigma_n(\theta_0) = \fisher(\theta_0)$.
\end{theorem}
The proof is in Section~\ref{sec:proof-theo-estim}, where a explicit expression of $\Sigma_n(\theta_0)^{-1}$ is supplied.

\subsection{Generalization}
\label{sec:generalization}

It is possible to extend the statistical framework introduced in the previous sections to a large class of reinforced random walks on graphs. For instance, let~${G=(\Vset,\NOEset)}$ be a locally finite non oriented graph, with~$\Vset$  the set of its vertices and $\NOEset\subset \big\{\{i,j\}:\; i,j\in \Vset,\,i\neq j\big\}$ the set of its non oriented edges. We denote by $x\sim y$, if $\{x,y\}\in\NOEset$. Let $\Nset^\Vset$ be the set of integer vectors indexed on~$\Vset$. We define~${X=(X_n)_{n\ge 0}}$ a random walk on~$G$, i.e. a sequence of vertices, such that, for all~${n\ge 0}$, $X_{n+1}\sim X_n$. The vector $Z_n\in\Nset^\Vset$ is such that, for all vertex~$x\in\Vset$, $Z_n(x)$ is the number of times the walk~$X$ has visited~$x$ up to time~$n$. Let $\Theta\subset\Rset^d$, $d\ge 1$, we suppose that the walk~$X$ is vertex-reinforced and that there exists a choice function $f:\Theta\times\Vset\times\Nset^\Vset\to[0,1]$ and a parameter~${\theta\in\Theta}$ such that, for all~${n\ge 0}$ and~${x\in\Vset}$,
\begin{eqnarray*}
\pr(X_{n+1}=x|X_0,\cdots,X_n) &=& f(\theta,x,Z_n)\1{x\sim X_{n}}\;.
\end{eqnarray*}

For instance, the choice function can be similar to the one proposed by \mbox{J.-L.} Deneubourg \citedef{Deneubourg1990},
for $n\ge 0$ and~${x\in\Vset}$,
\begin{eqnarray*}
\pr(X_{n+1}=x|X_0,\cdots,X_n) &=& \frac{\big(c+Z_n(x)\big)^\alpha}{\displaystyle\sum_{y\sim X_n} \big(c+Z_n(y)\big)^\alpha}\1{x\sim X_{n}},
\end{eqnarray*}
where $\Theta = (0,+\infty)^2$ and $\theta=(\alpha,c)$.

The statistical framework introduced for the general urn model is easily adaptable to this vertex reinforced random work. Under some adequate regular conditions, it would be not difficult to prove the consistency and the asymptotic normality of the MLE and the WLSE by establishing a theorem similar of the general result proved in Section~\ref{sec:MCE}.

\section{Application to an ethological problem}
\label{sec:bio-theory}

In~1990, J.-L.~Deneubourg et al. used a particular urn model to reproduce the sequences of
consecutive choices made by ants at a fork~\citedef{Deneubourg1990}. Let replace the urn filled with
balls of two colors by a fork with two branches. Drawing a ball and adding a ball of the same color
in the urn is equivalent to an ant choosing a branch and reinforcing it with pheromone by going
throw it. Then the probability to draw a red ball depending on the previous draws is equal to the
probability to choose the right branch depending on the previous passages.

We first introduce the model proposed by J.-L.~Deneubourg et al.~in the statistical framework
described in the previous section. Futher, we provide a description of the model behavior depending
on the parameter value. We also supply an ethological interpretation of the parameter. We then prove
that it is impossible to estimate the parameter on a single path. Finally we report the study of the
estimator accuracy that we perform on simulated data.

\subsection{Formalization and results}
\label{sec:theorie-interpretation}
\paragraph{Behavioral assumptions} We first introduce the hypotheses assumed on the ant behavior.
\begin{enumerate}
\item \label{hyp:same-reinf} Each ant regularly deposits a constant amount of pheromone as it walks.
\item \label{hyp:same-react} The ants are strictly identical which means that every ant has the
  same reaction to the same amount of pheromone.
\item \label{hyp:non-evap} Pheromone trails do not evaporate during the experiment.
\item \label{hyp:unique-pass} Each ant reaches the fork alone, chooses a branch and leaves the
  bifurcation by crossing only once into the chosen branch without passing through or reinforcing
  the non-chosen branch.
\end{enumerate}
Under these assumptions, the quantity of pheromone laid on each branch is proportional to the number
of passages through it. In path formation modeling, these assumptions are commonly made. However,
because of the inter-individual variability in ant behavior, the first two assumptions are
unrealistic. For instance, the pheromone perception noise implies that each ant could detect a
different signal from the same quantity of pheromone. These assumptions are an approximation of the
real ant behavior. The implicit hypothesis here is that the inter-individual variability is small
enough to consider that all ants are identical. For the third assumption, we suppose that the
persistence of the pheromone trails allows  to ignore the evaporation of the
pheromone. Experimental protocols are designed to make the four assumptions more acceptable by
choosing the ant species adequately and by placing them in an appropriate situation (see
Section~\ref{sec:ants}).

\paragraph{The model}
The random variable $X_k$, introduced previously, is the choice of the $k$-th ant going through the fork (1 for right and~0~for left). Consequently, for $k\geq1$, $Z_k$ is the number of passages through the right branch after~$k$~passages. For $\theta=(\alpha,c)\in(0,\infty)^2$ and all integers $0 \leq i \leq k$, we define choice function
\begin{align}
  \label{eq:f}
  f(\theta,i,k-i) = \frac{(c+i)^\alpha}{(c+i)^\alpha+ (c+k-i)^\alpha} \; .
\end{align}
We assume that the probability that an ant chooses the right branch at time~${k+1}$ given the first~$k$ choices is given by:
\begin{align}
\label{eq:choice-proba}
  \pr(X_{k+1}=1|\mathcal{F}_k) & = f(\theta,Z_k,k-Z_k) = \frac{(c+Z_k)^\alpha}{(c+Z_k)^\alpha+(c+k-Z_k)^\alpha} \; ,
\end{align}
where $c>0$ is the intrinsic attractiveness of each branch and $\alpha>0$ is the possible non-linearity of the choice. This process is an urn model which has been exhaustively investigated in the probabilistic literature. We recall here its main features.

\begin{theorem}
\label{theo:rrw-behavior}\
\begin{enumerate}[(i)]\setlength{\itemsep}{-1mm}
	\item \label{theo:rrw-behavior:a<1} If $\alpha<1$, then
		\begin{eqnarray*}
		\lim_{n\rightarrow\infty} \frac{Z_n}n = \frac{1}{2} \; a.s.
		\end{eqnarray*}
	\item \label{theo:rrw-behavior:a=1} If $\alpha=1$, then $Z_n/n$ converges almost surely to a random limit with a Beta$(c,c)$  distribution with density $x\to\Gamma^{-2}(c)\Gamma(2c) x^{c-1}(1-x)^{c-1}$ with respect to Lebesgue's measure on $[0,1]$, and $\Gamma$ is the Gamma function.
	\item \label{theo:rrw-behavior:a>1} If $\alpha>1$, then eventually only one branch will be chosen, i.e.
		\begin{eqnarray*}
		\exists x\in\{0,1\}\;,\ \exists n_0 \in \mathbb{N}\;,\ \forall n \ge n_0\;, \ X_n = x \;.
		\end{eqnarray*}
\end{enumerate}
\end{theorem}
The case $\alpha<1$ is due to \cite{Tarres2011}; the case $\alpha>1$ to \cite{Davis1990} and the
case $\alpha=1$ to \cite{Polya1931} (see \cite{Freedman1965} for an online access).

\paragraph{Ethological interpretation of the parameters and properties of the model}

The parameter~$\alpha$ characterizes the ant's differential sensitivity to the pheromone. When
$\alpha>1$, the ants can detect better and better increasing amounts of pheromones laid on each
branch and thus are more likely to choose the branch with the most pheromones. Moreover, after a
random but almost surely finite number of passages, one branch will eventually be selected, i.e. all
ants will afterwards choose this branch (see
Theorem~\ref{theo:rrw-behavior}-\eqref{theo:rrw-behavior:a>1}). In the opposite case, when
$\alpha<1$, the ants are less able to perceive the differences between the amounts of pheromones
laid on each branch as these amounts increase. This minimization effect is so strong that the
proportion of passages on each branch converges to~$1/2$ (see Theorem~\ref{theo:rrw-behavior}-\eqref{theo:rrw-behavior:a<1}). It is important to note that,
when~${\alpha\ne1}$, the asymptotic behavior of the proportion of passages through each branch
depends only on~$\alpha$ and not on~$c$. Furthermore the larger~$\alpha$~is or the closer~$\alpha$ is
to zero, the faster these effects will happen.

The role of~$c$ is clearer when~\eqref{eq:choice-proba} is rewritten as follows:
\begin{align*}
  \pr(X_{k+1}=1|\mathcal{F}_k) & = \frac{\left(1+{Z_k}/{c}\right)^\alpha}
  {\left(1+{Z_k}{c}\right)^\alpha + \left(1+({k-Z_k})/{c}\right)^\alpha} \; .
\end{align*}
The parameter~$c$ is the inverse of the reinforcement incrementation and thus can be interpreted as
the inverse of the attractiveness (or the strength) of the pheromones laid at each
passage. Consequently when $\alpha$ is neither very close to zero nor very large, $c$ has a strong
short term influence. When~$c$ is small compared to~$1$, the first passage strongly reinforces the
first chosen branch. Thus during the first few passages, a branch will be highly favored even
if~${\alpha<1}$ (in which case the branches will eventually be uniformly crossed). When~$c$~is
large, the first passages weakly reinforce the chosen branches. Then if~${\alpha>1}$ (in which case
a branch will eventually be selected), a large number of passages must be observed before the clear
emergence of a preference. Naturally, the larger~$c$ is or the closer~$c$ is to zero, the longer these
effects will be seen.

When $\alpha=1$, the asymptotic behavior of the passage proportion is determined by $c$ (see Theorem~\ref{theo:rrw-behavior}-\eqref{theo:rrw-behavior:a=1}). As~$c$
grows from zero to infinity, the limiting distribution of~$Z_n/n$ (as~${n\to\infty}$) evolves
continuously from two Dirac point masses at~0 and~1 to a single Dirac mass at~1/2. To illustrate
this point, we show in Figure \ref{beta_distribution} the density of the Beta distribution
for~$c=0.1$ and~$c=10$. We make some further comments.
\begin{itemize}
\item If $c<1$, a strong asymmetry in the choices of the branches appears. One branch is eventually
  chosen much more frequently than the other. Furthermore as~$c$ tends to~$0$, the Beta distribution
  tends to the distribution with two point masses at~$0$ and~$1$. This limit case corresponds to the
  situation in which a branch is selected, i.e.~$\alpha>1$.
\item If $c=1$, the limiting distribution is uniform on~$[0, 1]$.
\item If $c>1$, $Z_n/n$ appears to be much more concentrated around~$1/2$. This is similar to what
  is observed in the case~${\alpha<1}$.
\end{itemize}

\begin{figure}[h!]
\centering
\includegraphics[width=0.5\textwidth]{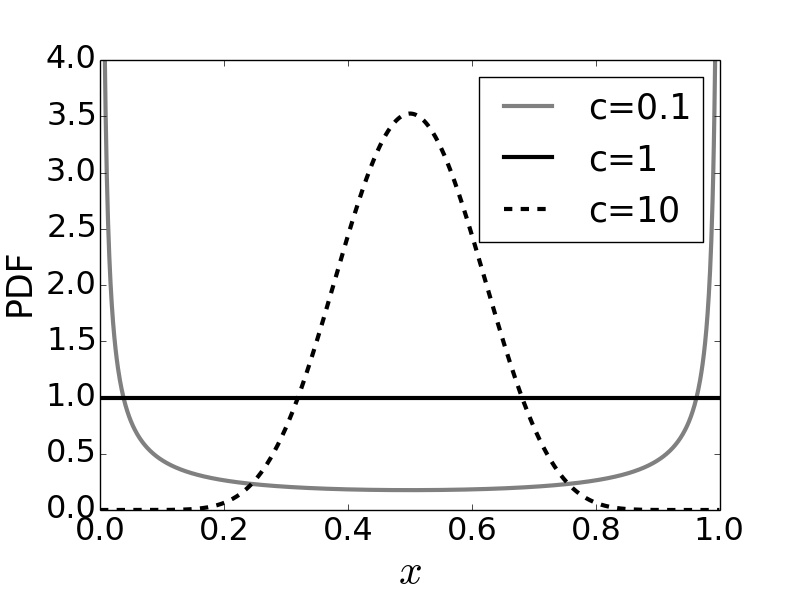}
\caption{Graph of the Beta distribution for parameters $(c,c)$ with $c=0.1$, $c=1$ and~${c=10}$}
\label{beta_distribution}
\end{figure}

To summarize, the model possesses four phases: fast and slow
selections and fast and slow uniformizations. These phases are delimited by two phase transitions: a
discontinuous one between~${\alpha>1}$ (branch selection) and~${\alpha<1}$ (branch uniformization)
with a critical state~${\alpha=1}$ and a smooth one between~${c<1}$ (strong pheromone deposits)
and~${c>1}$ (weak pheromone deposits). These properties are summarized in the phase diagram in
Figure~\ref{fig:phase_diag}.  When~$c$~is small, it is very likely that one branch will be favored
during the first passages, thus the empirical distribution of the choices resembles the Beta
distribution with a small~$c$ (the grey solid line in Figure~\ref{beta_distribution}). As the number
of experiments increases, the shape of the empirical distribution will be closer and closer to its
limit: a Dirac mass at~$1/2$ if~$\alpha<1$ and two Dirac masses at~0 and~1 if~${\alpha>1}$. When~$c$~is large, the earlier passages do not show any preference between the branch. Again, when the number
of experiments increases, the asymptotic behavior is progressively revealed.

To date, the most commonly used behavioral state in the model is the slow selection of a branch
\citedef{Deneubourg1990,Beckers1992a}. But at least two other states are interesting. The fast
uniformization can describe the case where none of the branches are preferred. The slow
uniformization could reproduce the saturation phenomenon. There exists a threshold concentration of
pheromone upon which ants can no longer detect the pheromone concentration variations
\citedef{Pasteels1987}. In experiments involving many ants, one can first observe the favorization of a
branch. But when this branch is saturated (its attractiveness stops increasing), ants go more and
more through the other branch, whose attractiveness still increases. Eventually, the two branches are
uniformly chosen.
\clearpage
\begin{remark}
  \label{rem:identi-pb}
  In the context of an estimation procedure, the similar effects of~$\alpha$ and~$c$ (favorization/selection of a branch or not) induce an identifiability issue. Indeed, we observe a
  finite number of choices and consequently we only see the short term behavior.  We have seen that
  the favorization of a branch in the first passages could be due to a pair of parameter values
  $(\alpha,c)$ with $c$ small compared to~$1$ and~$\alpha$ close to~$1$ or to a pair $(\alpha,c)$
  with $c$ close to~$1$ and $\alpha$ large compared to~$1$. On the other hand, if the first passages
  are nearly uniform on the two branches, it could be the result of $c$ close to~$1$ and $\alpha$
  small, or of $c$ large and~$\alpha$ close to~$1$.
\end{remark}
Thus, we can expect that the estimation of the parameters will be difficult when both parameters
contribute to the same effect, e.g.~$\alpha$ large and~$c$ small (fast selection of one branch)
or~$\alpha$ small and~$c$ large (no selection); and also when the parameters have competing effects:
very small~$c$ and~${\alpha<1}$, or very large~$c$ and~${\alpha>1}$. The statistical procedure that
we introduce in this paper partially circumvents this difficulty, since it focuses
on the transition probabilities rather than on the general shape of a curve, which is what
calibration methods do. This will be illustrated in Section~\ref{sec:simulation}.
\begin{figure}[h!]
\center
\includegraphics[width=0.9\textwidth]{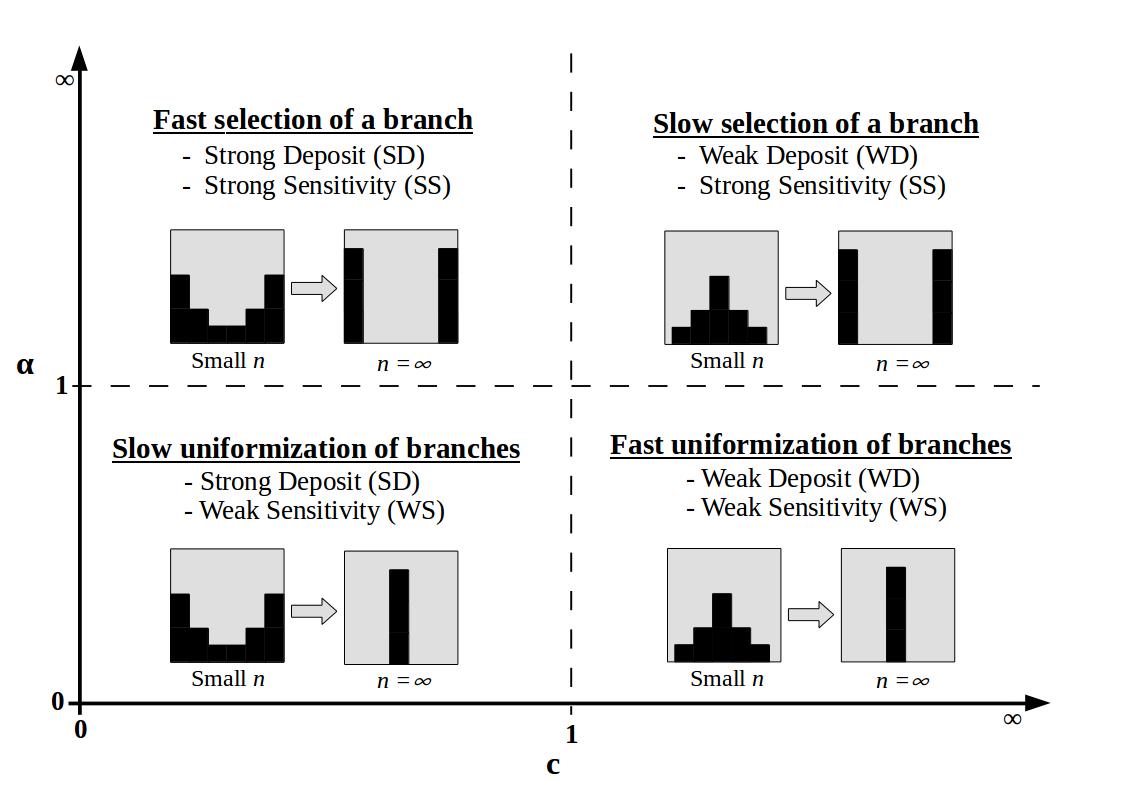}
\caption{Phase diagram of the model. The graphs in the shaded boxes show the shape of the empirical
  distribution of $Z_n/n$ for small $n$ (left) and its limiting distribution (as~${n\to\infty}$,
  right).}
\label{fig:phase_diag}
\end{figure}

\subsection{Parameter estimation}
\label{sec:parametric}

We start with the maximum likelihood estimator, that is $\hat\theta_N=(\hat\alpha_N,\hat{c}_N)$ defined by~(\ref{eq:def-mle}).  The
Fisher information matrix has the following expression.
\begin{align}
 \label{eq:fisher-rrw}
  \fisher(\theta)=\sum_{k=0}^{n-1} \sum_{i=0}^k \pr(Z_k = i ) f(\alpha,c,i,k-i)\bar{f}(\alpha,c,i,k-i)  \mcj(\alpha,c,i,k-i) \;,
\end{align}
with $\theta=(\alpha,c)$ and for $0\le i\le k\le n-1$,
\begin{align*}
  \mcj(\alpha,c,i,k-i)=\begin{pmatrix}
    \log^2\left(\frac{c+i}{c+k-i}\right) & \log\left(\frac{c+i}{c+k-i}\right)  \frac{\alpha (k-2i)}{(c+i)(c+k-i)}  \\
    \log\left(\frac{c+i}{c+k-i}\right) \frac{\alpha
      (k-2i)}{(c+i)(c+k-i)}
    & \frac{\alpha^2 (k-2i)^2}{(c+i)^2(c+k-i)^2} \\
  \end{pmatrix} \; .
\end{align*}
It is important to note that the Fisher information matrix is not diagonal. Thus the estimation of each parameter has an effect on the estimation of the other.

\begin{corollary}
\label{theo:mle-rrw} Let $\Theta$ be a compact subset of $(0,\infty)^2$ which contains $(\alpha_0,c_0)$.  Then the maximum likelihood estimator $\hat\theta_N$ is consistent and asymptotically normal and efficient, i.e.~$\sqrt{N}(\omle-\theta_0)$ converges weakly to $\mathcal{N}(0,\fisher^{-1}(\theta_0))$.
\end{corollary}

As mentioned above, we also use weighted least squared estimators defined by (\ref{eq:def-wlse}) for several different weight sequences $w_N$, such that $w_N(i,k-i)$ converges almost surely to~$w_0(i,k-i)>0$, for all $0\le k\le n-1$ and $0\le i\le k$.

\begin{corollary}
\label{theo:wlse-rrw}
Let $\Theta$ be a compact subset of~$(0,\infty)^2$ which contains $(\alpha_0,c_0)$ and assume that~$w_N$ converges almost surely to positive weights~$w_0$. Then the weighted least squared estimator $\wlsemodtheta$ is consistent and asymptotically normal, i.e. ${\sqrt{N}(\wlsemodtheta-\theta_0)}$ converges weakly to $\mathcal{N}(0,\Sigma_n(\theta_0))$, where $\Sigma_n(\theta_0)$ is a positive definite covariance matrix. It is efficient, if~${w_N(i,k-i)=p_N^{-1}q_N^{-1} a_N(i,k-i)}$, for all $0\le k\le n-1$ and $0\le i\le k$.
\end{corollary}

The proofs of these corollaries are in Section~\ref{sec:proof-rrw}. They are an application of Theorems~\ref{theo:consistence-clt-emv} and~\ref{theo:wlse}. Since Assumption~\ref{hypo:modele-regulier}-\eqref{item:regularite} obviously holds, it remains only to check \eqref{item:identifiability} and \eqref{item:invertibilite} of Assumption~\ref{hypo:modele-regulier}.

\subsection{Estimation on a single path}
\label{sec:fisher-single-path}

The main feature of the binary choice model for $\alpha_0>1$ is that only one branch will be crossed
eventually. It seems clear then that a statistical procedure based on only one path (one sequence of
choices) cannot be consistent, since no new information will be obtained after one branch is
eventually abandoned. This intuition is true and more surprisingly, it is also true in the case
$\alpha_0=1$. This is translated in statistical terms in the following theorem. Let $\ell_n$ denote
the log-likelihood based on a single path of length $n$ and $\dot\ell_n$ its gradient. The model is
regular, so the Fisher information is $\var_\theta(\dot\ell_n(\theta))$.
\begin{theorem} 
\label{theo:fisher-one-path}\
	\begin{enumerate}[(i)]
	\item \label{item:fisher-alpha=1} If $\alpha_0=1$ and $c_0>0$, then $\lim_{n\to\infty} \fisher(c_0)<\infty$. 
	\item \label{item:likelihood-alpha<1} If $\alpha_0<1$, $n^{-1}\ell_n(\theta_0)\to -\log2$.
	\item \label{item:likelihood-alpha>1} If $\alpha_0>1$, then $\ell_n(\theta_0)$ converges almost surely to a random variable as~${n\to\infty}$.
	\end{enumerate}
\end{theorem}
The proof is in Section~\ref{sec:proof-theo-one-path}. Statement~\eqref{item:fisher-alpha=1} means that, when $\alpha_0=1$, the Fisher information is bounded. This implies that the parameter $c_0$ cannot be estimated on a single path. This also implies that the length $n$ of each path should be taken as large as possible (theoretically infinite) in order to minimize the asymptotic variance of the estimators. Statements~\eqref{item:likelihood-alpha<1} and~\eqref{item:likelihood-alpha>1} imply that the maximum likelihood estimator is inconsistent, since the likelihood does not tend to a constant.

\subsection[Simulations]{Simulation experiment}
\label{sec:simulation}

In order to assess the quality of the estimators proposed, we have made a short simulation study. For several pairs $(\alpha,c)$, we have simulated 1000 experiments of $N=50$ paths of length $n=100$ (recall that $n$ is the number of ants going through the bifurcation). These are reasonable values in view of the practical experiments with actual ants. We compare the  performance of the maximum likelihood estimator $\omle$ (MLE) defined in~(\ref{eq:def-mle}) and of the weighted least squares estimator $\wlsemodtheta$ (WLSE) defined in~(\ref{eq:def-wlse}) with the weights~${w_N(i,k-i)=a_N(i,k-i)}$ defined in~(\ref{eq:empiriques}). The asymptotically efficient WLSE with the weights $w_N(i,k-i)=a_N(i,k-i) p_N(i,k-i)^{-1}q_N(i,k-i)^{-1}$ provides a severely biased estimation of $\alpha$ and always estimates a very small value of~$c$ with a very small dispersion. This is caused by the fact that the empirical $p_N$ and $q_N$ vanish frequently, so that the weights are infinite. We will not report the study for this estimator.

\paragraph{The theoretical standard deviation}

We first evaluate numerically some values of the theoretical standard deviations of both estimators for several values of~$\alpha$ and~$c$.  We have chosen arbitrary values of $\alpha$ and $c$ in the range~${0.5,2}$. We have also chosen values of $\alpha$ and $c$ which correspond to those found in the literature cited and to those that we have estimated in the real life experiment described in Section~\ref{sec:ants}. These results are reported in Table~\ref{tab:std-mse} and in Figure~\ref{fig:a_std} and their features are summarized in the following points.

\begin{itemize}
\item As theoretically expected, the asymptotic variance of the MLE, which is the Fisher information bound, is smaller than the variance of the WLSE, but the ratio between the variances of the two estimators is never less than one fourth. Moreover, their overall behavior is similar.
\item The variance of the estimators of $\alpha$ is smaller when both parameters do not contribute to the same effect. The worst variance is for $\alpha$ large and~$c$~small, that is when the values of both parameters imply fast selection of a branch. The variance tend to infinity when~$\alpha$ tends to infinity.
\item The variance of the estimators of $c$ increases with $c$ and tends to infinity when $\alpha$ tends to~0 and to~$\infty$.
\item These effects are explained by the fact that the coefficients of the Fisher information matrix tend to zero when $\alpha$ tends to zero, except the coefficient corresponding to $\alpha$. See Formula~(\ref{eq:fisher-rrw}).
\end{itemize}

\begin{figure}[h!]
\begin{minipage}[c]{0.5\textwidth}
	\centering
	\subfigure[TSD of $\alpha$]{\includegraphics[width=\textwidth]{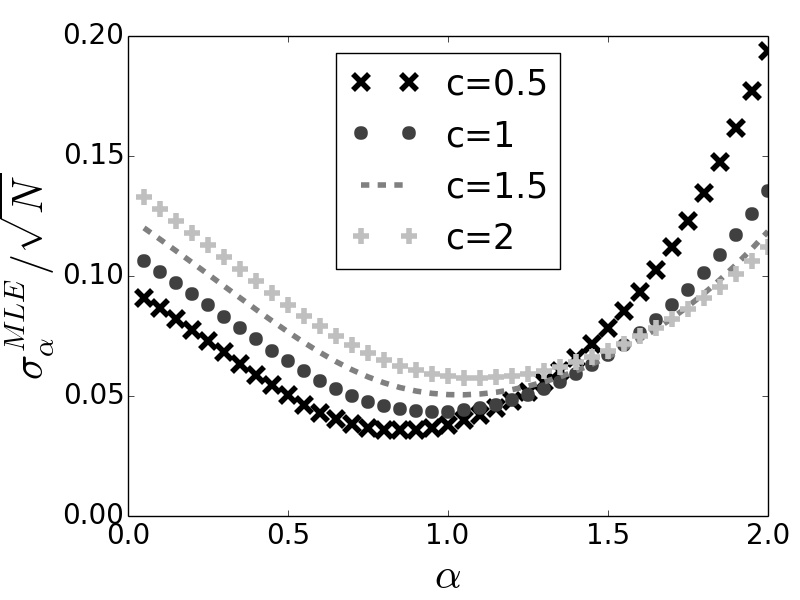}}
\end{minipage} \hfill
\begin{minipage}[c]{0.5\textwidth} 
	\centering
	\subfigure[TSD of $c$]{\includegraphics[width=\textwidth]{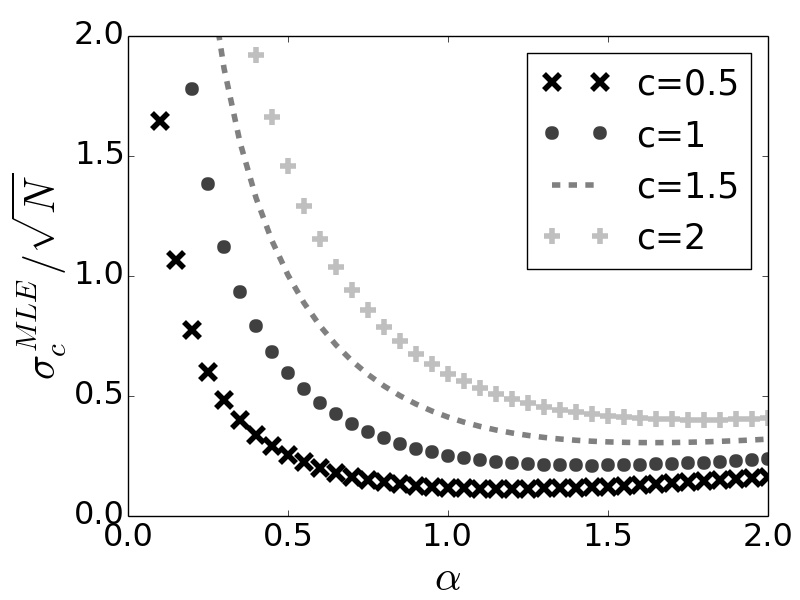}}
\end{minipage}
\caption{\label{fig:a_std} Theoretical standard deviation for $N=50$ paths of length~${n=100}$ of the MLE of $\alpha$ (a) and of $c$ (b), for $\alpha$ in $(0,2]$ and fixed values of $c$. Theoretical standard deviation of the WLSE has the same shape, but with higher convergence speed as $\alpha$ tend to $0$ or $\infty$.}
\end{figure}

\paragraph{Performance of the estimators} Recall that we have simulated 1000 experiments, each of
$N=50$ paths of length $n=100$. Because of the length of the computations, each MLE was computed
only 500 times. Table~\ref{tab:std-mse} reports root mean squared error (MSE) of both estimators
based on the simulated data for the same values of the parameters and their features are summarized
in the following points.

\begin{itemize}
\item For most values of the parameters, the MSE are close to the theoretical standard deviation.
\item The MSE increase significantly when $c$ is large or when both parameters contribute to the same effect.
\item This increase is more noticeable for the WLSE than for the MLE.
\item This increase is in part due to the skewness of these estimators. For some values of the parameters, both estimators tend to overestimate the parameters.
\item For the MLE, the MSE is much larger in the case of non selection than in the case of selection where the empirical performance of the MLE nearly matches the theoretical value.
\item These effects are always stronger for the estimation of $c$ than for the estimation of $\alpha$.
\end{itemize}

This degraded performance for some specific or extreme values of the parameters is in part due to numerical issues.
\begin{itemize}
\item In the case where selection of a branch is fast, many of the empirical weights used to compute the WLSE vanish, and the least squares method uses very few points to fit the curve. The MLE is not affected by this problem.
\item In the case where both parameters concur to non selection, the probability of choosing one branch converges very fast to 1/2, and thus the experiment brings very little information. This affects both the MLE and the WLSE, and in addition, many of the empirical weights vanish so the WLSE is even less efficient.
\end{itemize}

The degraded performance may also be caused by  the identifiability problem explained in Remark~\ref{rem:identi-pb}, i.e. the similar effects of the two parameters makes the estimation more difficult.

\begin{table}[h!]
\centering
\caption{\label{tab:std-mse} Theoretical standard deviations (TSD) for~${N=50}$ paths of length~${n=100}$ and square root of the mean squared errors (MSE) for~${500}$ (for the MLE) or~${1000}$ (for the WLSE) simulated experiments of~${N=50}$ paths of length~${n=100}$. All figures of this table must be mutiply by~${0.01}$.}\small
\begin{tabular}{|c||c:c|c:c||c:c|c:c|}
\hline
\multirow{3}{*}{($\alpha,c$)}	&\multicolumn{4}{c||}{$\alpha$}&	\multicolumn{4}{c|}{$c$}\\
\cline{2-9}	&\multicolumn{2}{c|}{MLE}& \multicolumn{2}{c||}{WLSE}& \multicolumn{2}{c|}{MLE}&\multicolumn{2}{c|}{WLSE}\\
\cline{2-9}	&TSD	&$\sqrt{MSE}$	&TSD		&$\sqrt{MSE}$	&TSD		&$\sqrt{MSE}$	&TSD		&$\sqrt{MSE}$\\\hline
$(0.5,0.5)$	&$5.02$	&$5.60$	&$6.25$	&$7.60$	&$25.4$ 	&$54.1$	&$31.7$	&$58.8$	\\
$(0.5,1.0)$	&$6.45$	&$7.39$	&$8.56$	&$10.2$	&$59.8$	&$137$	&$79.4$	&$169$	\\
$(0.5,2.0)$	&$8.80$	&$16.9$	&$12.9$	&$356$	&$146$	&$651$	&$214$	&$29500$	\\
$(1.0,0.5)$	&$3.81$	&$3.90$	&$4.97$	&$6.25$	&$11.8$	&$12.4$	&$15.9$	&$21.0$	\\
$(1.0,1.0)$	&$4.34$	&$4.98$	&$5.91$	&$7.34$	&$25.2$	&$31.3$	&$34.9$	&$47.3$	\\
$(1.0,2.0)$	&$5.83$	&$5.80$	&$8.72$	&$9.38$	&$59.3$	&$64.9$	&$89.1$	&$111$	\\
$(1.5,0.5)$	&$7.83$	&$7.85$	&$12.8$	&$19.4$	&$12.3$	&$13.0$	&$18.7$	&$24.6$	\\
$(1.5,1.0)$	&$6.69$	&$6.59$	&$10.8$	&$15.1$	&$21.0$	&$21.4$	&$33.1$	&$40.3$	\\
$(1.5,2.0)$	&$6.88$	&$7.24$	&$11.5$	&$13.9$	&$41.7$	&$45.9$	&$68.8$	&$78.4$	\\
$(2.0,0.5)$	&$19.4$	&$26.1$	&$41.8$	&$2690$	&$16.4$	&$19.8$	&$28.6$	&$1410$ 	\\
$(2.0,1.0)$	&$13.5$	&$14.2$	&$28.7$	&$985$	&$23.9$	&$26.0$	&$44.3$	&$988$	\\
$(2.0,2.0)$	&$11.2$	&$11.7$	&$23.1$	&$35.4$	&$40.7$	&$43.4$	&$77.2$	&$101$	\\
\hdashline
$(2.0,20.0)$	&$35.5$	&$58.8$	&$129$	&$124$	&$794$	&$1420$	&$2880$	&$3240$	\\
$(2.6,60.0)$	&$166$	&$864$	&$1230$	&$3590$	&$5780$	&$32500$	&$43200$	&$142000$	\\
$(1.1,3.0)$	&$7.20$	&$7.80$	&$11.9$	&$11.3$	&$89.9$	&$111$	&$148$	&$201$	\\
$(1.1,7.0)$	&$13.8$	&$17.4$	&$29.5$	&$28.3$	&$297$	&$433$	&$633$	&$758$	\\\hline
\multicolumn{9}{|c|}{\bf{All figures must be multiply by $0.01$}}\\\hline
\end{tabular}
\end{table}
\paragraph{Bootstrap confidence intervals}
Since the asymptotic variance depends on the unknown parameters, we have computed the pivotal
Bootstrap 95\% confidence intervals for the parameters based on one simulation of $N=50$ paths of
length $n=100$ and a Bootstrap sample size of $500$ (see \cite{Wasserman2004}, Section~8.3, for
details on this method). We have compared these Bootstrap intervals with the corresponding
Monte-Carlo intervals, based on 500 simulations (see Table~\ref{tab:boot_IDC}). The match is nearly perfect for the MLE for
$\alpha$, but as before, the performance is poorer for the estimation of $c$. The intervals for $c$
are noticeably skewed to the right but always contain the true value. For further comparison, we
only show here the results corresponding to the values of the parameters estimated in the real life
experiment reported below and those corresponding to values found in the earlier literature.

\begin{table}[h!]
\centering
\caption{\label{tab:boot_IDC} Monte-Carlo 95\% confidence intervals for $500$ simulated experiment of $N=50$ paths of length~${n=100}$ and Bootstrap 95\% confidence intervals for one simulated experiment of $50$ paths of length~$100$}
\begin{tabular}{|c|c|c|c|c|c|c|c|c|c|} \hline
\multirow{2}{*}{}	&\multirow{2}{*}{($\alpha,c$)}	&\multicolumn{2}{c|}{IDC $95\%$ for $\alpha$ }	&\multicolumn{2}{c|}{IDC $95\%$ for $c$} \\ 
\cline{3-6}	&	&Monte-Carlo		&Bootstrap	&Monte-Carlo	&Bootstrap   \\\hline
MLE	&$(2.0,20)$	&$(1.50,3.38)$	&$(1.90,5.10)$	&$(9.90,55.5)$	&$(19.1,94.4)$	\\
	&$(2.6,60)$	&$(1.23,29.3)$	&$(1.12,38.7)$	&$(15.9,1053)$	&$(19.9,1637)$	\\
    &$(1.1,3.0)$	&$(0.98,1.29)$	&$(1.01,1.23)$	&$(1.81,6.22)$	&$(1.34,3.99)$	\\
\hdashline
WLSE	&$(2.0,20)$	&$(1.26,4.85)$	&$(1.06,3.99)$	&$(5.92,88.1)$	&$(2.76,69.3)$	\\
    &$(2.6,60)$	&$(0.74,88.3)$	&$(0.24,87.7)$	&$(3.86,3754)$	&$(0.17,4132)$	\\
    &$(1.1,7.0)$	&$(0.75,1.83)$	&$(0.55,2.56)$	&$(1.68,30.3)$	&$(1.48,66.0)$	\\\hline
\end{tabular}
\end{table}

\clearpage
\section[Ants]{Real life experiment with ants}
\label{sec:ants}

In this section, we apply the previous estimators on data from  a path selection
experiment by a colony of ants.

\subsection{Experiment description} 
\label{sec:exp_desc}
This experiment was done in the Research Center on Animal Cognition (UMR~5169) of Paul Sabatier
University Toulouse under the supervision of Guy Theraulaz, Hugues Chat\'e and the first author. A
small laboratory colony (approximately 200~workers) of Argentine ants {\it Linepithema humile} was
starved for two days before the experiment. During the experiment, the colony had access to a fork
carved in a white PVC slab, partially covered by a Plexiglas plate (see
Figure~\ref{fig:setup_Y}). The angle between the branches was~$60^\circ$. The fork galleries had a
$0.5$~cm square section. The entrance of the maze was controlled by a door. Food was never present
during the experiment. The maze was initially free of any pheromone trail.

\begin{figure}[h!]
  \centering
  \includegraphics[width=0.8\textwidth]{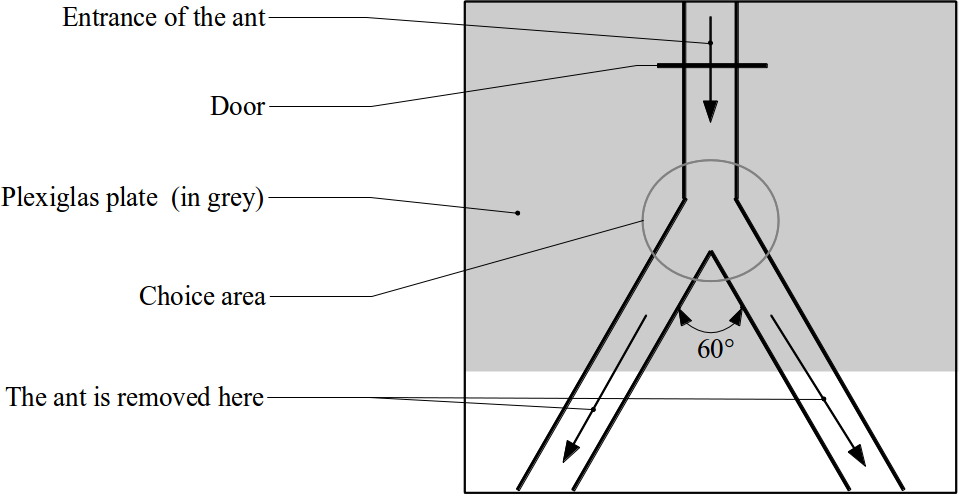}
  \caption{The experimental set-up: a fork carved in a white PVC slab, partially covered by a Plexiglass plate.}
  \label{fig:setup_Y}
\end{figure}

Each trial ($N=50$) consisted in introducing separately each ant to the entrance of the fork
(see Figure~\ref{fig:setup_Y}) one at a time. Once inside, an ant must choose between the left or the
right branch of the fork. As soon as the ant had made a choice and stepped into one branch, it was
removed from the set-up and another ant was introduced. All the choices were recorded and a trial
ended when 100~ants had passed through the fork.

This experimental protocol was designed to strengthen the behavioral assumptions described in
Section~\ref{sec:theorie-interpretation}. Any return to the fork is forbidden so that we can
consider that each ant passed only one time. There was never more than one ant in the set-up. This
implies that each ant in the maze received no other cue about the previous passages than the
pheromone that was been laid. The species {\it Linepithema humile} was in part chosen to justify the
assumption of identical pheromone deposits. Indeed, these ant may deposit regularly the same type of
pheromone on their trajectory \citedef{VanVorhisKey1982,Aron1989}. All ants were prepared the same way
before the experiments to increase the credibility of the assumption stating that each ant behaved
by the same way. The length of the experiments was limited to stay close to the half-life duration
of the pheromone trails \citedef{Jeanson2003}.

\subsection{Data representation}
\label{sec:data_repres} 
Figures \ref{fig:traj_exp} shows the 50 paths of length $n=100$, that is, $50$ choice sequences of
100 ants that went through the bifurcation. The paths are represented as random walks with
increment~$+1$ when the right branch is chosen, and~$-1$ when the left one is chosen.  In less than
ten experiments, a branch seemed to be selected, whereas in the others, selection of a branch was
not obvious. Figure \ref{fig:hist_exp} shows the histogram of the distribution of $Z_{100}/100$,
that is the final proportion of the choices of the right branch. There is no clear visual evidence
that $\alpha>1$ as it is claim in the literature \citedef{Deneubourg1990,Vittori2006,Garnier2009}.

\begin{figure}[h!]
\begin{minipage}[c]{0.48\textwidth}
\subfigure[\label{fig:traj_exp} The 50 paths of $n=100$ ants choosing either left~(+1) or right~(-1)]{\includegraphics[width=\textwidth,clip]{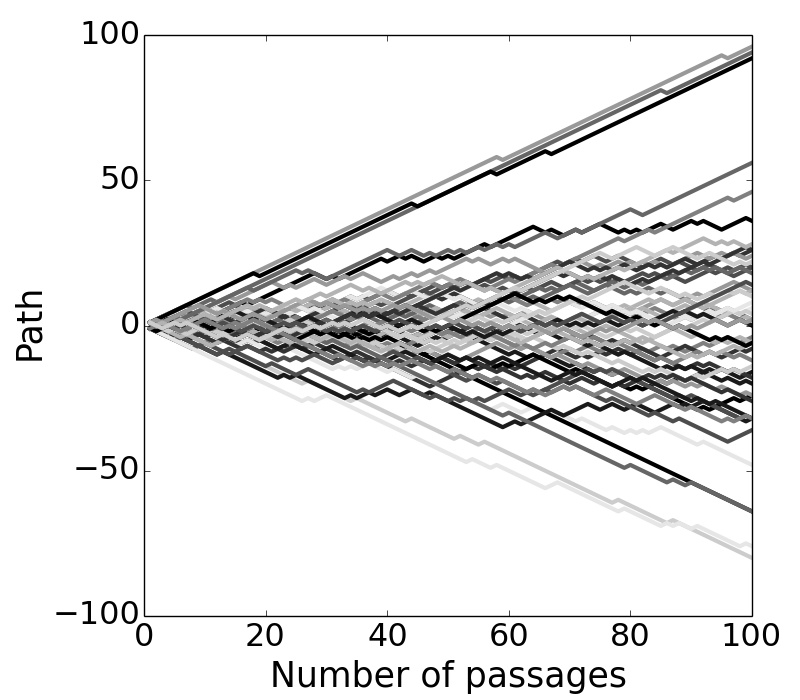}}
\end{minipage} \hfill
\begin{minipage}[c]{0.48\textwidth}
\subfigure[\label{fig:hist_exp} Histogram of the final proportion of right passages ($Z_{100}/100$)]{\includegraphics[width=\textwidth,clip]{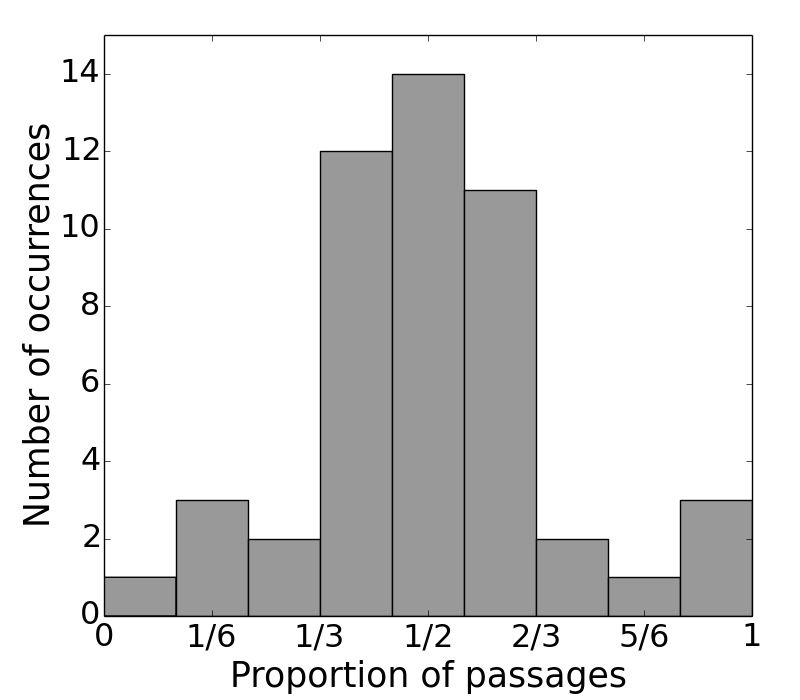}}
\end{minipage}
\caption{Data representation}
\end{figure}

\subsection{Parameter estimation}
\label{sec:para_est}
Several values of these parameters have been proposed in the applied
literature. \cite{Deneubourg1990} proposed $\alpha=2$, $c=20$ and more recently \cite{Garnier2009}
suggested $\alpha=2.6$ and $c=60$. It must be noted however that these values are not obtained by a
statistical method but by the calibration of a curve to a plot.  Therefore, these methods do not
lead to confidence intervals. Moreover, a calibration method has an inherent risk of over fitting,
because of the identifiability problem explained in Remark~\ref{rem:identi-pb}. As illustrated in
Figure~\ref{fig:phase_diag}, if for instance~$\alpha$ and~$c$ are both small, then both branches
will be asymptotically equally chosen, but paths of finite length $n$ might be misleading and the
calibration will suggest values of~$\alpha$ and~$c$ corresponding to the selection of a branch. The
statistical procedure is based on the dynamics of the process and is thus less prone to this type of
error. Nevertheless, we will see that our results do not contradict those of \cite{Deneubourg1990}
and \cite{Garnier2009}, but complement them.

\begin{table}[h!]
\centering
\caption{\label{tab:exp} The MLE and the WLSE for the 50 paths of real ants and their Boostrap~$95\%$ confidence intervals}
\begin{tabular}{|c|c|c|c|c|c|c|c|c|c|} \hline
	& $\hat\alpha$	&Bootstrap 95\% CI & $\hat c$	&Bootstrap 95\% CI 	\\\hline
MLE	&$1.07$	&$(0.80,1.99)$	&$3.26$	&$(1.14,23.0)$	\\\hline
WLSE	&$1.10$	&$(0.62,3.81)$	&$6.91$	&$(0.94,85.4)$	\\\hline
\end{tabular}
\end{table}

Table \ref{tab:exp} shows the results of the maximum likelihood estimation and the weighted least
squares estimation. Both estimates of $\alpha$ are close to 1.1 and the estimates of $c$ are
between~3 and~7. The $95\%$ confidence intervals are slightly larger than the simulated ones (see
Table~\ref{tab:boot_IDC}). This increased variability may be due to the extreme paths which seem to
show a very fast selection of one branch (see Figures \ref{fig:traj_exp} and
\ref{fig:hist_exp}). This may suggest that the ants did not have the same behavior and that the
distribution of $Z_{100}/100$ could be a mixture of two distributions.

For both methods, the 95\% Bootstrap confidence intervals of $\alpha$ contain the value~1. More
precisely, as shown in Figure~\ref{fig:ellipse}, approximately 1/3 of the bootstrap parameters gives
weak pheromone deposits ($c>1$) and a weak differential sensitivity ($\alpha<1$), which means that
branches are eventually uniformly crossed. In almost all the others cases, we conclude for weak
pheromone deposits ($c>1$) and a strong differential sensitivity ($\alpha>1$), which means that a
branch will be eventually, though slowly, selected. In only a few cases do the estimators give
strong pheromone deposits ($c<1$), but a weak differential sensitivity ($\alpha<1$), which means
that a branch is chosen more than the other at the beginning of the experiment, but branches are
eventually uniformly crossed. Finally, there are no values which imply both strong pheromone
deposits ($c<1$) and a strong differential sensitivity ($\alpha>1$). Therefore, we can conclude that
pheromone deposits are weak with a good confidence but we cannot confidently decide for~$\alpha$.

The values obtained by \cite{Deneubourg1990} ($\alpha=2$, $c=20$) and more recently by
\cite{Garnier2009} ($\alpha=2.6$, $c=60$) are both in the confidence intervals for the WLSE found in
Table~\ref{tab:exp}. But the values of $\alpha$ suggested by these authors are out of the 95\%
confidence interval for the MLE. Thus these parameters, which decide for a slow branch selection,
are no more likely than a parameter set which would yield non selection of a path.

Figure~\ref{fig:ellipse} illustrates the fact that the two estimators are strongly positively correlated. There seems to be two cutoff values for $c$: if $\hat c^*>8$, then~${\hat\alpha^*>1}$, and if $\hat{c}^*<1.5$, then~${\hat{\alpha}^*<1}$. The above mentioned values reported by
\cite{Deneubourg1990} and \cite{Garnier2009} exhibit these features: they both have $c>8$ and
$\alpha>1$ and $\alpha$ increase with~$c$.

If we fix the value of~$c$ and estimate only~$\alpha$, then the 95\% Bootstrap confidence intervals
for~$\alpha$ are smaller. Figure \ref{fig:profil} shows the estimated values of~$\alpha$ and the
confidence intervals as functions of the fixed value of~$c$. We see that if $c$ is greater than~$6$
for the MLE (or than~$12$ for the WLSE), then the confidence intervals of~$\alpha$ lie entirely
above~$1$. Furthermore if~$c$~is less than~$0.8$ for the MLE (or than~$2$ for the WLSE), then the
confidence intervals of~$\alpha$~lie entirely under~$1$. This shows that if the deposits are weak
enough, i.e.~$c>12$, we can conclude that a slow selection of a branch will occur with
probability~1. On the other hand, if the deposits are strong enough, i.e.~$c<0.8$, we can conclude
that branches will eventually be uniformly crossed with probability~1.

\begin{figure}[h!]
\begin{minipage}[c]{0.49\textwidth}
\subfigure[MLE]{\includegraphics[width=\textwidth,clip]{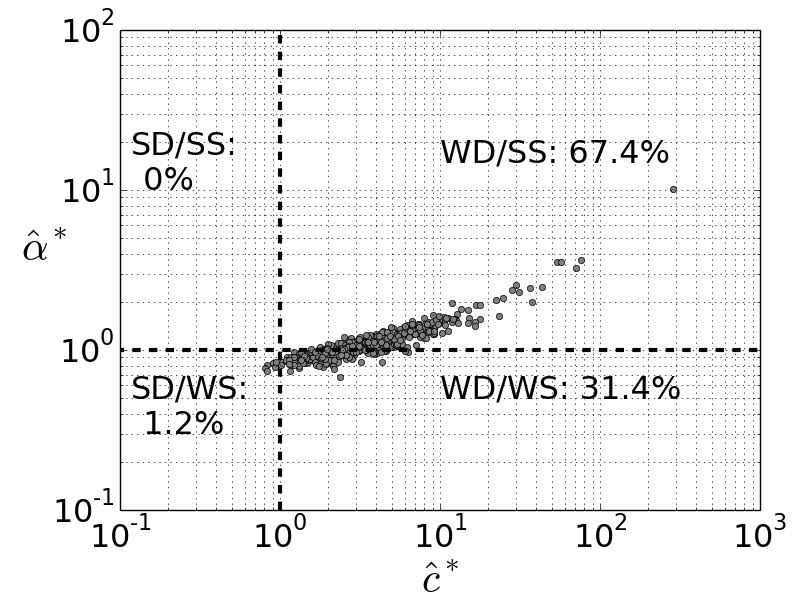}}
\end{minipage} \hfill
\begin{minipage}[c]{0.49\textwidth}
\subfigure[WLSE]{\includegraphics[width=\textwidth,clip]{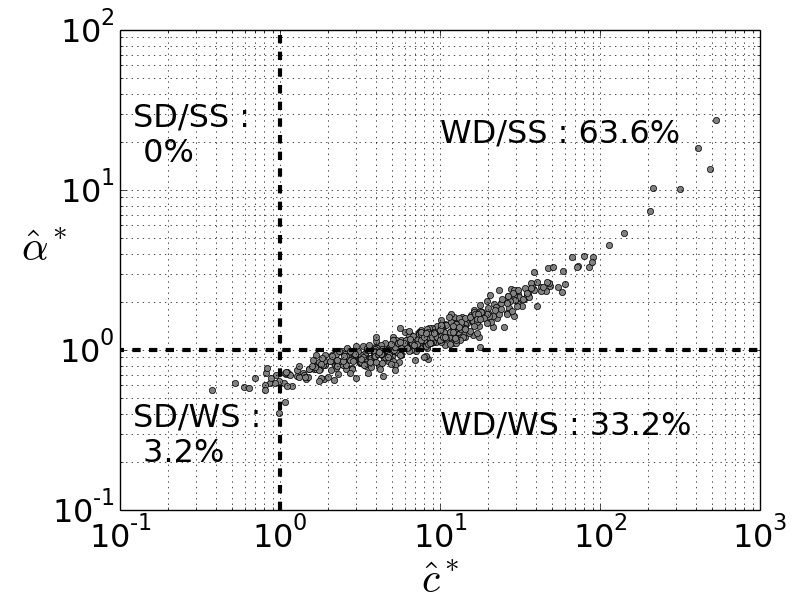}}
\end{minipage}
\caption{\label{fig:ellipse} Log-log scatterplots of the estimates $(\hat\alpha^*,\hat{c}^*)$ for the 500 Bootstrap samples for the MLE (a) and the WLSE (b) for the 50 paths of real ants.}
\end{figure}

\begin{figure}[h!]
\begin{minipage}[c]{.49\textwidth}
	\centering 
	\subfigure[MLE]{\includegraphics[width=\textwidth]{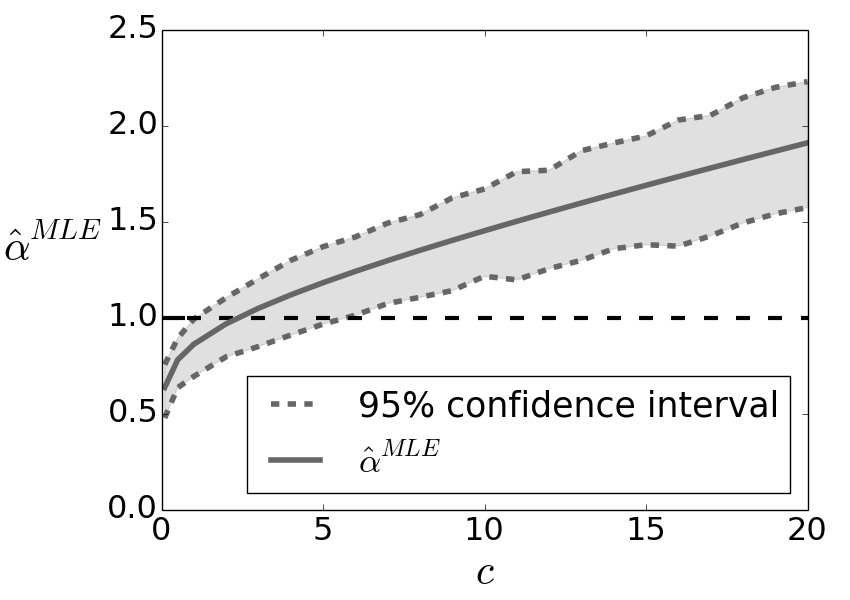}}
\end{minipage} \hfill
\begin{minipage}[c]{.49\textwidth}
	\centering
	\subfigure[WLSE]{\includegraphics[width=\textwidth]{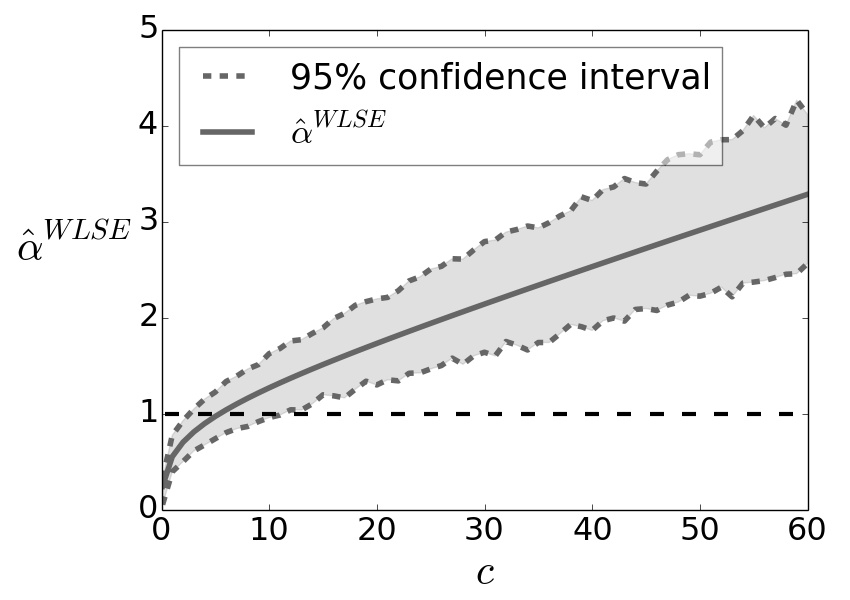}}
\end{minipage}
\caption{\label{fig:profil} Graph of the estimates $\hat\alpha$ and their Bootstrap 95\% confidence interval for the MLE (a) and the WLSE (b) for the 50 paths of real ants as a function of fixed value of $c$.}
\end{figure}

\section{Concluding remarks}
\label{sec:concluding}

In the literature no parameter estimation methods for reinforced random walks can be found. To
partially fill this void, this article proposes a statistical framework to estimate the parameter
of a general two-colored urn model. We define the maximum likelihood estimator (MLE) and the
weighted least squares estimators (WLSE) for the parameter of this model and prove their consistency
and their asymptotically normality under some usual regularity conditions. The proof lies on a general
result for a large class of estimators called minimum contrast estimators. The MLE is asymptotically
efficient, but can be difficult (lengthy) to compute, which can be an issue specially while using
Bootstrap algorithms. The WLSE is a suitable alternative. Moreover this estimator is popular among
practitioners.

We apply this statistical tools to the problem of path selection by an ant colony. To this
purpose, we performed experiments with actual ants to collect data. The experiment consisted of
introducing one hundred ants into a $Y$ shaped device, one at a time, and observing their successive
choices. We also consider the particular urn model introduced by \cite{Deneubourg1990} to describe
this phenomenon. This urn has two parameters,~$\alpha$ and $c$, which have distinct biological
interpretations, but contribute to the same effect: either selection of a branch or uniformization
of the choices. The parameter $c$ influences the short term behavior, whereas $\alpha$ determines
the asymptotic behavior. Consequently the model exhibits four phases which are illustrated by
Figure~\ref{fig:phase_diag}. The case most commonly considered in the literature is the case of slow
selection, which corresponds to~${\alpha>1}$ and~${c>1}$: the ants will eventually always choose the
same branch, but this selection will take a long time. For instance, \cite{Deneubourg1990} provides
the values~${\alpha=2}$ and~${c=20}$. However other phases can be relevant to describe the ant
behavior. For instance the fast uniformization, corresponding to~${\alpha<1}$ and~$c>1$, can model
the less likely but not negligible case in which ants do not select a branch.

After assessing the accuracy of the MLE and the WLSE on simulated data, we estimate the value
of~$\alpha$ and~$c$ with the two estimators. We also evaluate confidence regions by Bootstrap
proceeding. The estimated values of~$\alpha$ and~$c$ ranged between~1.1 and~3 and between~3 and~7,
respectively. This tends to imply that slow selection of a branch will occur. However, the Bootstrap
sample gives a confidence level of~${65\%}$ for the hypothesis of slow selection, while the
hypothesis of fast uniformization has a confidence of~${35\%}$.

This low level of confidence for the commonly assumed slow selection phase might be explained by
technical reasons. The number of experiments~(50) is relatively small; increasing the number of
replicas will reduce the confidence regions. Moreover the competition between the parameters for the
same effect induces an identifiability issue. For instance the apparent preference of a branch may
be due to~${\alpha>1}$ or~$c$ small with respect to~1. Therefore, the model, which is biologically
relevant, is statistically difficult to estimate. Indeed, for an ethological study, discriminating
the ant pheromone sensitivity from the pheromone deposit strength is meaningful. But for a
statistical procedure, the similarity of effect of the two parameters scales down the estimation
performance.

However, the uncertainty may not come from an inefficiency of the statistical procedure, but from
shortcomings of the ethological hypotheses. Indeed, the estimated confidence intervals computed from
the experimental data are larger than the ones computed from the simulated data (for similar
parameter values). Moreover the assumption that the inter-individual variability is negligible is
strong. For instance, it may be necessary to consider that the pheromone deposit varies at each
passage, i.e. that $c$ is random.

These ethological considerations will be further discussed in a forthcoming paper which will analyze
more elaborated experimental designs. The ants will be observed while freely evolving in a network
with several nodes. In addition of a data analysis, we will model the experiment with a reinforced
random walk on a finite graph for which we have provided probabilistic results~\citedef{LeGoff2015}.
The statistical methodology introduced in this paper will be extended to a larger class of
reinforced random walks.

\section{Proofs}
\label{sec:proofs}

\subsection{Distribution of $Z_k$, for $k\in\Nset$}
\label{sec:P(Zk=i)}

In order to compute the distribution of $Z_k$, we introduce some notation. Let~$\suite_k$ be the set of sequences of length $k+1$ of integers $i_0,\dots,i_k$ such that $i_0=0$ and $i_j-i_{j-1}\in\{0,1\}$ for $j=1,\dots,k$. For $i\leq k$ let $\suite_k(i) = \{(i_0,\dots,i_k) \in \suite_k\mid i_k=i\}$. Then we have
\begin{equation}
\label{eq:proba-Zk}
\pr(Z_k=i) = \sum_{(i_0,\dots,i_k)\in\suite_k(i)} \prod_{q=0}^{k-1} f_0(i_q,q-i_q)^{i_{q+1}-i_q} (1-{f}_0(i_q,q-i_q))^{1-i_{q+1}+i_q} \; . 
\end{equation}

\subsection{A central limit theorem for the empirical conditional probabilities}

For $0 \leq i \leq k \leq n-1$, recall the definition of $a_N(i,k-i)$ and $p_N(i,k-i)$ in~(\ref{eq:empiriques}) and that~${\bar{f}_0(i,k-i)=1-f_0(i,k-i)}$.

\begin{lemma}
\label{lem:clt-empirique}
$\{\sqrt{N}(p_N(i,k-i)-f_0(i,k-i)), 0 \leq i \leq k \leq n-1\}$ converges weakly to a Gaussian vector with diagonal covariance matrix $\Gamma_0$ with diagonal elements 
\begin{align}
\gamma_0(i,k-i) = \frac{f_0(i,k-i)\bar{f}_0(i,k-i)}  {\pr(Z_k=i)} \; .  \label{eq:diagonale}
\end{align}
\end{lemma}

\begin{proof}
Define $b_N(i,k-i) = N^{-1} \sum_{j=1}^N \1{Z_k^{j}=i}X_{k+1}^j$, the empirical estimate of $b(i,k-i) = \pr(Z_{k}=i,X_{k+1}=1)$ and $a(i,k-i) = \esp[a_N(i,k-i)] = \pr(Z_k=i)$. Write then
\begin{align*}
p_N(i,k-i)-f_0(i,k-i) =& \frac{b_{N}(i,k-i)-b(i,k-i)}{a_N(i,k-i)}\\
	& -\frac{b(i,k-i)}{a_N(i,k-i)a(i,k-i)} (a_N(i,k-i)-a(i,k-i)) \; .
\end{align*}
Since the paths $(Z_1^j,\dots,Z_n^j)$, $1 \leq j \leq N$ are i.i.d., the multivariate central limit holds for the sequence of $2n(n-1)$ dimensional vectors $\{(b_{N}(i,k-i)-b(i,k-i),a_N(i,k-i)-a(i,k-i)),0\leq i \leq k \leq n-1\}$. The proof is concluded by tedious computations using the Markov property, which we omit.
\end{proof}

\begin{remark}
\label{rem:mle-tautologie}
We can prove that the covariance matrix $\Gamma_0$ is diagonal by a statistical argument. If we consider the tautological model $\{f(i,k-i),0 \leq i \leq k \leq n-1\}$, i.e. $\theta=f$ and~$f_0$ is the true value. Then the likelihood is
\begin{multline*}
L_N(f) = \sum_{k=0}^{n-1} \sum_{i=0}^k a_N(i,k-i) \{p_N(i,k-i)\log f(i,k-i) 
\\+ q_N(i,k-i)\log\bar{f}(i,k-i)\} \; ,
\end{multline*}
where $\bar{f}(i,k-i)=1-{f}(i,k-i)$. Thus we see that ${\{p_N(i,k-i), 0\leq i \leq k\leq n-1\}}$ is the maximum likelihood estimator of~$f_0$. This model is a regular statistical model, thus~${\sqrt{N}(p_N-f_0)}$ converges weakly to the Gaussian distribution with covariance matrix~${I_n^{-1}(f_0)}$, where $I_n(f)$ is the Fisher information matrix of the model. It is easily seen that~${I_n(f_0)}$ is the $n(n-1)$ dimensional  diagonal matrix with diagonal elements given by~(\ref{eq:diagonale}).
\end{remark}

\subsection{A general result for minimum contrast estimators}
\label{sec:MCE}

Theorems~\ref{theo:consistence-clt-emv} and~\ref{theo:wlse} are a consequence of the general result we prove in this section. More precisely we demonstrate the consistency and the asymptotic normality of a general estimator of which the MLE and the WLSE are particular cases.

For $0 \leq i \leq k \leq n-1$, recall the definition of $a_N(i,k-i)$,  $p_N(i,k-i)$ in~(\ref{eq:empiriques}) and that $\bar{f}(\theta,i,k-i)=1-f(\theta,i,k-i)$. Let $w_N(i,k-i)$, $0\leq i \leq k \leq n-1$ be a sequence of random weights and let $G$ be function defined on $[0,1]\times(0,1)$. Define the empirical contrast function by 
\begin{align*}
\mbw_N(\theta) = \sum_{k=0}^{n-1} \sum_{i=0}^k w_N(i,k-i) G(p_N(i,k-i),f(\theta,i,k-i)) \; .
\end{align*}
For instance, choosing $G(p,q) = - p \log q - (1-p)\log(1-q)$ and $w_N(i,k-i)=a_N(i,k-i)$ yields
\begin{align*}
\mbw_N(\theta)  = &- \sum_{k=0}^{n-1} \sum_{i=0}^k a_N(i,k-i) \big\{p_N(i,k-i) \log f(\theta,i,k-i) \\
&+ q_N(i,k-i) \log \bar{f}(\theta,i,k-i)\big\} \\
	= & -N^{-1} L_N(\theta) \; ,
\end{align*}
so that minimizing $\mbw_N$ is equivalent to maximizing the likelihood~$L_N$, defined in~\eqref{eq:def-LN}. Choosing $G(p,q)=(p-q)^2$  yields the weighted least squares contrast function~$W_N$, defined in~\eqref{eq:def-WN}. We now define the minimum contrast estimator of~$\theta_0$ by
\begin{align*}
\mcemodtheta=\arg\min_{\theta\in\Theta} \mbw_N(\theta) \; .
\end{align*}

In order to prove the consistency and asymptotic normality of $\mcemodtheta$, we make the following assumptions on $G$ and on the weights $w_N(i,k-i)$.  Let~$\partial_2G$ and~$\partial_2^2G$ denote the first and second derivatives of $G$ with respect to its second argument.
\begin{hyp}
\label{hypo:contrast}
The function $G$ is non negative, twice continuously differentiable on $[0,1]\times(0,1)$ with $G(p,q)-G(p,p)>0$ if $p\ne q$,  $\partial_2G(p,p)=0$ and~${\partial_2^2G(p,p)>0}$.
\end{hyp}

\begin{hyp}
\label{hypo:poids}
For all $0 \leq i \leq k \leq n-1$, $w_N(i,k-i)$ converge almost surely to $w_0(i,k-i)$ and $w_0(i,k-i)>0$.
\end{hyp}

\begin{theorem}
\label{theo:Big}
If Assumptions~\ref{hypo:modele-regulier}-\eqref{item:regularite}, \ref{hypo:modele-regulier}-\eqref{item:identifiability}, \ref{hypo:contrast} and~\ref{hypo:poids} hold, then $\mcemodtheta$ is consistent. If moreover~$\theta_0$ is an interior point of $\Theta$ and Assumption~\ref{hypo:modele-regulier}-\eqref{item:invertibilite} holds, then $\sqrt{N}(\mcemodtheta-\theta_0)$ converges weakly to a Gaussian distribution with zero mean.
\end{theorem}
The exact expression of the variance is given in the proof. 

\begin{proof}
Under Assumption~\ref{hypo:contrast}, the strong law of large numbers shows that~$\mbw_N(\theta)$ converges almost surely to
\begin{align*}
\mbw (\theta) =  \sum_{k=0}^{n-1} \sum_{i=0}^k w_0(i,k-i) G(f_0(i,k-i),f(\theta,i,k-i)) \; .
\end{align*}

Assumptions~\ref{hypo:modele-regulier}-\eqref{item:identifiability} and~\ref{hypo:contrast} ensure that $\theta_0$ is the unique minimum of $\mbw$. Indeed, $G(p,q)>0$ if $p\neq q$ and $G(p,p)=0$. Thus, $\mbw$ is minimized by any value of $\theta$ such that $f(\theta,i,k-i)=f(\theta_0,i,k-i)$. By Assumption~\ref{hypo:modele-regulier}-\eqref{item:identifiability}, this implies $\theta=\theta_0$.

Moreover the convergence of~$\mbw_N$ to $\mbw$ is uniform, since $\Theta$ is compact and the function~$f$ is twice continuously differential with respect to~$\theta$, its first variable. This yields the consistency of $\mcemodtheta$. For the sake of completeness, we give a brief proof. Since~$\theta_0$ minimizes~$\mbw$ and~$\mcemodtheta$ minimizes~$\mbw_N$, we have
\begin{align*}
0 & \leq  \mbw(\mcemodtheta) - \mbw(\theta_0)\\
	& = \mbw(\mcemodtheta) - \mbw_N(\mcemodtheta) + \mbw_N(\mcemodtheta)- \mbw_N(\theta_0)  + \mbw_N(\theta_0) - \mbw(\theta_0)  \\
	& \leq \mbw(\mcemodtheta) -\mbw_N(\mcemodtheta) + \mbw_N(\theta_0) - \mbw(\theta_0) \leq 2 \sup_{\theta\in\Theta} |\mbw_N(\theta)-\mbw(\theta)| \; .
\end{align*}
Since~$\theta_0$ is the unique minimizer of~$\mbw$, for~$\epsilon>0$, we can find~$\delta$ such that if~${\theta\in\Theta}$ and~$\|\theta-\theta_0\|>\epsilon$, then $\mbw(\theta)- \mbw(\theta_0) \geq \delta$. Thus
\begin{align*}
\pr(\| \hat\theta_N - \theta_0\| > \epsilon) & \leq \pr( \mbw(\hat\theta_N) -\mbw(\theta_0) \geq \delta)\\
& \leq \pr \left(2\sup_{\theta\in\Theta} |\mbw_N(\theta)-\mbw(\theta)| \geq \delta \right) \to 0 \; .
\end{align*}
The central limit theorem is a consequence of the consistency and Lemma~\ref{lem:clt-empirique}. A first order Taylor extension of $\dot{\mbw}_N(\theta)$ at $\theta_0$ yields
\begin{align*}
0 & = \dot{\mbw}_N(\mcemodtheta) = \dot{\mbw}_N(\theta_0) + \ddot{\mbw}_N(\tilde\theta_N) (\mcemodtheta-\theta_0) \; ,
\end{align*}
where $\tilde\theta_N\in[\theta_0,\mcemodtheta]$.  Setting $\dot{f}_0(i,k-i) = \dot{f}(\theta_0,i,k-i)$, we have 
\begin{align*}
\dot{\mbw}_N(\theta_0) = \sum_{k=0}^{n-1} \sum_{i=0}^k w_N(i,k-i) \partial_2G(p_N(i,k-i),f_0(i,k-i)) \dot{f}_0(i,k-i) \; .
\end{align*}

Let $\partial_{12}^2 G$ be the mixed second derivative of $G$.  Note that $$\partial_2G(f_0(i,k-i),f_0(i,k-i))=0\:.$$ Thus, by the delta-method \cite[see][Theorem~3.3.11]{Dacunha-castelle1986} and since $w_N$ converges almost surely to $w_0$, we obtain that $\sqrt{N} \dot{\mbw}_N(\theta_0)$ converges weakly towards
\begin{align*}
\sum_{k=0}^{n-1}\sum_{i=0}^k w_0(i,k-i) \partial_{12}^2G (f_0(i,k-i),f_0(i,k-i)) \Lambda_0(i,k-i) \dot{f}_0(i,k-i) \; ,
\end{align*}
where $\Lambda_0(i,k)$ are independent Gaussian random variables with zero mean and variance~$\gamma_0(i,k)$ defined in~\ref{eq:diagonale}. Equivalently, $ \sqrt{N} \dot{\mbw}_N(\theta_0)$ converges weakly to a Gaussian vector with zero mean and covariance matrix $H(\theta_0)$ defined by
\begin{multline*}
H(\theta_0) = \sum_{k=0}^{n-1}\sum_{i=0}^k w_0^2(i,k-i) \{\partial_{12}^2 G(f_0(i,k-i),f_0(i,k-i)) \}^2\\
\times\gamma_0(i,k) \dot{f}_0(i,k-i) (\dot{f}_0(i,k-i))' \; .
\end{multline*}
By the law of large numbers, $\ddot{\mbw}_N(\theta)$ converges almost surely to $\ddot{\mbw}(\theta)$ and this convergence is also locally uniform. Thus, $\ddot{\mbw}_N(\tilde\theta_N)$ converges almost surely to~$\ddot{\mbw}(\theta_0)$. Using again the fact that $\partial_2G(p,p)=0$, we obtain
\begin{multline*}
\ddot{\mbw}(\theta_0) = \sum_{k=0}^{n-1}\sum_{i=0}^k w_0(i,k-i) \partial_2^2G (f_0(i,k-i),f_0(i,k-i))\\
\times\dot{f}_0(i,k-i) (\dot{f}_0(i,k-i))' \; .
\end{multline*}
Denote for brevity $g(i,k-i) = w_0(i,k-i) \partial_2^2G (f_0(i,k-i),f_0(i,k-i))$. Then, for any $u\in\Rset^d$, we have
\begin{align}
\label{eq:pos_def}
u \ddot{\mbw}(\theta_0)u' & = \sum_{k=0}^{n-1} \sum_{i=0}^k g(i,k-i) \left( \sum_{s=1}^d u_s \partial_sf(\theta_0,i,k-i) \right)^2 \; .
\end{align}
By assumption~\ref{hypo:poids}, $g(i,k-i)>0$ for all $0\leq i \leq k \leq n-1$, thus~\eqref{eq:pos_def} is zero only if for all~${k=0,\dots,n-1}$ and $i=0,\dots,k$, we have~$\sum_{s=1}^d u_s \partial_s f(\theta_0,i,k)=0$. By Assumption~\ref{hypo:modele-regulier}-\eqref{item:invertibilite}, this is possible only if $u_s=0$ for all $s=1,\dots,d$. Thus~$\ddot{\mbw}(\theta_0)$ is positive definite.

We can now conclude that for large enough $N$, $\ddot{\mbw}_N(\tilde\theta_N)$ is invertible and we can write
\begin{align*}
\sqrt{N}(\mcemodtheta-\theta_0) = - \ddot{\mbw}_N^{-1}(\tilde\theta_N) \sqrt{N} \dot{\mbw}_N(\theta_0) \; .
\end{align*}
The right hand side converges weakly to the Gaussian distribution with zero mean and covariance matrix $\ddot{\mbw}^{-1}(\theta_0)H(\theta_0)\ddot{\mbw}^{-1}(\theta_0)$.
\end{proof}

\subsection{Proofs of theorems~\ref{theo:consistence-clt-emv} and~\ref{theo:wlse}}
\label{sec:proof-theo-estim}

Theorems~\ref{theo:consistence-clt-emv} and~\ref{theo:wlse} are a consequence of Theorem~\ref{theo:Big}. 

\begin{lemma}
\label{lem:poids}
Assumption \ref{hypo:poids} holds for the weights $w_N(i,k-i) = a_N(i,k-i)$ and $w_N(i,k-i) = a_N(i,k-i)p_N^{-1}(i,k-i)q_N^{-1}(i,k-i)$, $0\le i\le k\le n-1$.
\end{lemma}

\begin{proof}
For all $0\le i\le k\le n-1$, $a_N(i,k-i)$ converges almost surely to~${\pr(Z_k=i)}$ and $a_N(i,k-i)p_N^{-1}(i,k-i)q_N^{-1}(i,k-i)$ to $\pr(Z_k=i)f_0(i,k-i)^{-1}\bar f_0(i,k-i)^{-1}$. Moreover Assumption~\ref{hypo:modele-regulier}-\eqref{item:regularite} implies that $f_0(i,k-i)>0$ and $\bar  f_0(i,k-i)>0$ for all $0 \leq i \leq k \leq n-1$. Using Formula~(\ref{eq:proba-Zk}), this in turn implies that $\pr(Z_k=i)>0$ for all $0 \leq i \leq k \leq n-1$.
\end{proof}

\begin{proof}[Proof of Theorem \ref{theo:consistence-clt-emv}]
As mentioned above, the maximum likelihood estimator minimizes the contrast function $\mbw$ written with the function $G(p,q)=-p\log q-{(1-p)}{\log(1-q)}$ and the weights $a_N(i,k-i)$. Thus the proof of Theorem~\ref{theo:consistence-clt-emv} consists in checking Assumption~\ref{hypo:contrast} and~\ref{hypo:poids} to apply Theorem~\ref{theo:Big}. Lemma~\ref{lem:poids} implies that Assumption~\ref{hypo:poids} holds.

The function $G$ considered here satisfies Assumption~\ref{hypo:contrast}. Indeed, for ${p,q\in(0,1)}$, define $$K(p,q) = G(p,q)-G(p,p) = p\log(p/q) + (1-p)\log((1-p)/(1-q)\:.$$ Remark that $K(p,q)$ is the Kullback-Leibler distance between the Bernoulli measures with respective  success probabilities $p$ and $q$.  Then it is well known that $K(p,q)>0$ except if $p=q$. Indeed, by Jensen's inequality,
\begin{align*}
K(p,q) \geq -\log(pq/p + (1-p)(1-q)/(1-p))=\log1 = 0 \; ,
\end{align*}
and by strict concavity of the log function, equality holds only if $p=q$. Moreover, $\partial_2G(p,q)=-p/q+(1-p)/(1-q)$ so $\partial_2G(p,p)=0$ and ${\partial_2^2G(p,p)=p^{-1}(1-p)^{-1}>0}$.
\end{proof}

\begin{proof}[Proof of Theorem~\ref{theo:wlse}]
Again, the proof consists in checking Assumption~\ref{hypo:contrast} and~\ref{hypo:poids} to apply Theorem~\ref{theo:Big}. The latter holds by virtue of Lemma~\ref{lem:poids} and Assumption~\ref{hypo:contrast} trivially holds for the function $G(p,q)=(p-q)^2$.
If ${w_N(i,k-i)} = {p_N^{-1}(i,k-i)}q_N^{-1}(i,k-i)a_N(i,k-i)$, then 
\begin{multline}
H(\theta_0) = 2 \ddot \mbw (\theta_0)= 4\fisher(\theta_0)\\
=4\sum_{k=0}^{n-1} \sum_{i=0}^k \frac{\pr(Z_k=i)}{ f_0(i,k-i) \bar{f}_0(i ,k-i)} \dot{f}_0(i,k-i) (\dot{f}_0(i,k-i))' \label{eq:wlse-hessien-efficient}\;.
\end{multline}
So that $\Sigma_n(\theta_0) = \ddot{\mbw}^{-1}(\theta_0)H(\theta_0)\ddot{\mbw}^{-1}(\theta_0)=\fisher^{-1}(\theta_0)$.
\end{proof}

If the weights are chosen as $w_N(i,k)=a_N(i,k)$, then $w_0(i,k-i) = \pr(Z_k=i)$ and 
\begin{align}
H(\theta_0) & = 4 \sum_{k=0}^{n-1} \sum_{i=0}^k \pr(Z_k=i) f_0(i,k-i) \bar{f}_0(i ,k-i) \dot{f}_0(i,k-i) (\dot{f}_0(i,k-i))'  \label{eq:H-wlse-pratique}\; , \\
\ddot{W}(\theta_0) & = 2 \sum_{k=0}^{n-1} \sum_{i=0}^k \pr(Z_k=i) \dot{f}_0(i,k-i) (\dot{f}_0(i,k-i))' \;. \label{eq:ddotW-wlse-pratique}
\end{align}

\subsection{Proofs of Corollaries~\ref{theo:mle-rrw} and~\ref{theo:wlse-rrw}}
\label{sec:proof-rrw}

Corollaries~\ref{theo:mle-rrw} and~\ref{theo:wlse-rrw} are a consequence of Theorem~\ref{theo:Big}. The assumptions on the weights~$w_N$ and on the functions~$G$ have been already verified in the previous section. We have to prove the Assumption~\ref{hypo:modele-regulier} on the choice function~$f$ defined in~\eqref{eq:f}. Hypothesis~\ref{hypo:modele-regulier}-\eqref{item:regularite} is obvious.

By elementary computations, we have, for $0 \leq i \leq k \leq n-1$,
\begin{align}
  f(\alpha,c,i,k-i) =f(\alpha_0,c_0,i,k-i)  
  & \Leftrightarrow \left(\frac{c+i}{c+k-i}\right)^\alpha =
  \left(\frac{c_0+i}{c_0+k-i}\right)^{\alpha_0} \nonumber  \\
  & \Leftrightarrow \frac{ \alpha}{\alpha_0} = \frac{\log(c_0+i) - \log(c_0+k-i)}
  {\log(c+i) - \log(c+k-i)} \; . \label{eq:equality-alpha}
\end{align}
Plugging the pairs $(i,k)=(0,1)$ and $(i,k)=(0,2)$ into~(\ref{eq:equality-alpha}) yields 
\begin{align*}
  \frac{\log(c_0) - \log(c_0+1)} {\log(c) - \log(c+1)} = \frac{\log(c_0) -
    \log(c_0+2)} {\log(c) - \log(c+2)} \; ,
\end{align*}
or equivalently
\begin{align}
  \label{eq:equality-c}
  \frac{\log(1+1/c_0)}{\log(1+2/c_0)} = \frac{\log(1+1/c)}{\log(1+2/c)} \; .
\end{align}
It is easily checked that the function $x\to \log(1+x)/\log(1/2x)$ is strictly increasing on~${(0,\infty)}$. Thus~(\ref{eq:equality-c}) implies that $c=c_0$. Plugging this equality into~(\ref{eq:equality-alpha}) yields~${\alpha=\alpha_0}$. This proves Assumption~\ref{hypo:modele-regulier}-\eqref{item:identifiability}.

We now prove that if $n\geq2$, the vectors $\{\partial_\alpha f(\theta_0,i,k-i), 0 \leq i \leq k \leq n-1\}$ and $\{\partial_c f(\theta_0,i,k-i), 0 \leq i \leq k \leq n-1\}$ are linearly independent in~$\Rset^{n(n-1)}$.  For~${0 \leq i \leq k \leq n-1}$, we have,
\begin{align*}
\partial_\alpha f(\alpha,c,i,k-i) & = f(\alpha,c,i,k-i){f}(\alpha,c,k-i,i) \log\left(\frac{c+i}{c+k-i}\right) \; ,  \\
\partial_c f(\alpha,c,i,k-i) & = f(\alpha,c,i,k-i) f(\alpha,c,k-i,i) \frac{\alpha(k-2i)}{(c+i)(c+k-i)} \; .
\end{align*}
Let $(u,v)\in\Rset^2$ and assume that for all $i,j\leq n-1$ such that $i+j\leq n-1$, it holds that
\begin{align*}
u \log\left(\frac{c_0+i}{c_0+j}\right) + v \frac {\alpha_0(j-i)} {(c_0+i)(c_0+j)}=0 \; .
\end{align*}
Replacing $(i,j)$ for instance successively by $(0,1)$ and $(0,2)$ yields
\begin{align*}
	\begin{cases}
	u \log\left(\dfrac{c_0}{c_0+1}\right) + v \dfrac {\alpha_0}{c_0(c_0+1)}=0 \; , \\
	u \log\left(\dfrac{c_0}{c_0+2}\right) + v \dfrac {2\alpha_0}{c_0(c_0+2)}=0 \; .
	\end{cases}
\end{align*}
If $(u,v)\ne(0,0)$, this implies
\begin{align*}
\dfrac{c_0 +2}{c_0}\log\left(\dfrac{c_0+2}{c_0}\right)+2\dfrac{c_0 +1}{c_0}\log\left(\dfrac{c_0+1}{c_0}\right) = 0 \; .
\end{align*}
By strict convexity of the function $x\to x\log x$ on $(0,\infty)$, this is impossible. Thus~${u=v=0}$ and Assumption~\ref{hypo:modele-regulier}-\eqref{item:invertibilite} holds.

\subsection{Proof of Theorem~\ref{theo:fisher-one-path}}
\label{sec:proof-theo-one-path}

\begin{proof}[Proof of Theorem~\ref{theo:fisher-one-path}, case $\alpha_0=1$]
In this case the model is P\'olya's urn, and we have  
\begin{align}
\label{eq:fisher-alpha=1}
\fisher(c) & = \sum_{k=0}^{n-1} \frac1{2c+k} \left\{ \esp \left[
\frac1{c+Z_k}\right] + \esp\left[ \frac1{c+k-Z_k} \right]- \frac{4}{2c+k} \right\} \; . 
\end{align}
The distribution of $Z_k$ is given by
\begin{align*}
\pr(Z_k=i) = \binom{k}{i} \frac{c(c+1)\cdots(c+i-1)\times c(c+1)\cdots(c+k-i-1)}{2c(2c+1)\cdots(2c+k-1)} \; .
\end{align*}
Thus, 
\begin{align*}
\esp \left[ \frac1{c+Z_k}\right] = \sum_{i=0}^k \binom{k}{i} \frac{c(c+1)\cdots(c+i-1)\times c(c+1)\cdots(c+k-i-1)}{2c(2c+1)\cdots(2c+k-1)} \frac1{c+i} \; .
\end{align*}
For any $c>0$, there exists constants $C_1<C_2$ such that, for all integers $h\geq1$, 
$$
C_1 h^c \leq \prod_{i=1}^h (1+c/i) \leq C_2h^c \; .
$$
Therefore, there exists a constant $C>0$ such that for all $k\geq1$,
\begin{align*}
\esp \left[ \frac1{c+Z_k}\right] & \leq C k^{-2} \sum_{i=1}^{k-1} \left(\frac{i}{k}\right)^{c-2} \left(1-\frac{i}{k}\right)^{c-1}
= \begin{cases} O(k^{-1}) \mbox{ if } c>1 \; , \\
    O(k^{-1} \log k) \mbox{ if } c=1 \; , \\
    O(k^{-c}) \mbox{ if } c<1 \; .
\end{cases}
\end{align*}
In all three cases, we obtain that the first series in~(\ref{eq:fisher-alpha=1}) is summable. By symmetry, the sum of the second expectations is also finite. 
\end{proof}

\begin{proof}[Proof of Theorem~\ref{theo:fisher-one-path}, case $\alpha_0<1$]
In this case, we know by Theorem~\ref{theo:rrw-behavior} that~$Z_n/n$ converges almost surely to 1/2. This implies that $f(\theta,Z_n,n-Z_n) $ converges almost surely to 1/2 for all $\theta$. By Cesaro's Lemma, this implies that~$n^{-1}\ell_n(\theta) \to -\log2$ a.s.
\end{proof}

\begin{proof}[Proof of Theorem~\ref{theo:fisher-one-path}, case $\alpha_0>1$]
Let $\Omega_1$ be the event that color 1 is eventually selected, which happens with probability 1/2 by Theorem~\ref{theo:rrw-behavior}. Then, on $\Omega_1$, ${Z_n/n\to1}$ and if~$k>\dernier_\infty$, then $X_{k+1} = 1$ and $Z_k=k-\total_\infty$. Thus for large enough $n$, the log-likelihood on one path becomes
\begin{multline*}
\ell_n(\theta) = \sum_{k=0}^{\dernier_\infty} X_{k+1} \log f(\theta,Z_k,k-Z_k) + (1-X_{k+1}) \log\{1-f(\theta,Z_k,k-Z_k)\} \\
+ \sum_{k=\dernier_\infty+1}^n \log f(\theta,k-\total_\infty,\total_\infty) \; .
\end{multline*}
As $k\to\infty$, for any $\alpha>0$, 
\begin{align*}
\log f(\theta,k-\total_\infty,\total_\infty) = - \log \left \{ 1 + \frac{(c+\total_\infty)^\alpha}{(c+k-\total_\infty)^\alpha} \right\} \sim - \frac{(c+\total_\infty)^\alpha}{(c+k-\total_\infty)^\alpha} \; .
\end{align*}
If $\alpha\leq1$ the series is divergent and thus $\lim_{n\to\infty} \ell_n(\theta)=-\infty$. If $\alpha>1$ then the series is convergent and thus, on  $\Omega_1$,
\begin{align*}
\lim_{n\to\infty} \ell_n(\theta) & = \sum_{k=0}^{\infty} X_{k+1} \log f(\theta,Z_k,k-Z_k) + (1-X_{k+1}) \log\{1-f(\theta,Z_k,k-Z_k)\}\\
	& = \sum_{k=0}^{\dernier_\infty} X_{k+1} \log f(\theta,Z_k,k-Z_k) + (1-X_{k+1}) \log\{1-f(\theta,Z_k,k-Z_k)\} \\
	& \hfill+ \sum_{\dernier_\infty+1}^\infty \log f(\theta,k-\total_\infty,Q_\infty) \; .
\end{align*}
This implies that $\arg\max_{\theta\in\Theta} \ell_n(\theta) = \arg\max_{\theta\in\Theta,\alpha>1} \ell_n(\theta)$ and that this argmax is a random variable which is a function of the whole path, and does not depend on the true value~$\theta_0$.
\end{proof}

\paragraph{Acknowledgment} 
We thanks Guy Theraulaz and Hugues Chat\'e for providing their material framework and their field expertise to allow the first author to collect the data of the Argentine ants experiments. These experiments are part of the project TRACES supported by the CNRS. They were done during two visits in April and July 2012 of the first author to the Centre de Recherches sur la Cognition Animale (CRCA, Centre de Recherches sur la Cognition Animale, UMR 5169, Paul Sabatier University, Toulouse), whose hospitality is gratefully acknowledged.

\bibliographystyle{DeGruyter}
\bibliography{bibLGS}

\end{document}